\documentclass[11pt]{article}
\usepackage{amsmath, amssymb, amsthm, mathtools}
\usepackage{geometry}
\usepackage{hyperref}
\usepackage{natbib}
\usepackage{setspace}
\usepackage{enumitem}
\usepackage{array}
\usepackage{booktabs}
\usepackage{graphicx}
\usepackage{tikz}
\usetikzlibrary{positioning}
\usepackage{pgfplots}
\pgfplotsset{compat=1.17}
\usepackage{microtype}
\theoremstyle{plain}
\newtheorem{theorem}{Theorem}
\newtheorem{proposition}[theorem]{Proposition}
\newtheorem{lemma}[theorem]{Lemma}
\newtheorem{corollary}[theorem]{Corollary}

\theoremstyle{definition}
\newtheorem{definition}[theorem]{Definition}
\newtheorem{assumption}[theorem]{Assumption}
\newtheorem{example}[theorem]{Example}

\theoremstyle{remark}
\newtheorem{remark}[theorem]{Remark}
\newtheorem*{theorem*}{Theorem}

\newcommand{\Pbb}{\mathbb{P}}
\newcommand{\Ebb}{\mathbb{E}}
\newcommand{\Var}{\mathrm{Var}}

\newcommand{\ind}{\mathbf{1}}
\newcommand{\dx}{\,\mathrm{d}}
\newcommand{\Bern}{\mathrm{Bern}}
\newcommand{\BetaD}{\mathrm{Beta}}

\title{Martingale Posterior Predictive Coherence:\\
Hausdorff Moment Hierarchy}
\author{
  Nicholas~G.\ Polson\thanks{Booth School of Business, University of Chicago,
    Chicago, IL 60637. Email: \texttt{ngp@chicagobooth.edu}.}
  \and
  Daniel Zantedeschi\thanks{Muma College of Business, University of South Florida,
    Tampa, FL 33620. Email: \texttt{danielz@usf.edu}.}
}
\date{\today}

\begin{document}
\maketitle

\begin{abstract}
\noindent
For an exchangeable Bernoulli sequence with de~Finetti mixing measure~$\Pi$,
the $k$-step predictive probability
$\Pbb(X_{n+1}=\cdots=X_{n+k}=0\mid\mathcal{F}_n)$ equals the posterior
expectation $\Ebb[(1-\theta)^k\mid\mathcal{F}_n]$.  By binomial
expansion, this depends on all posterior moments up to order~$k$.  We show
that the first moment alone is not sufficient to uniquely identify these
quantities: for $k\ge 2$, the mapping from posterior mean to $k$-step
predictive is set-valued.  The martingale posterior framework of Fong, Holmes,
and Walker (which constrains only the first conditional moment of the terminal
value) does not, in general, uniquely identify multi-step predictive
distributions.  Under any strictly proper scoring rule, the plug-in predictive
is strictly dominated by the Bayes predictive whenever the posterior is
non-degenerate.  A closure theorem establishes that a martingale posterior
determines all $k$-step predictives if and only if the conditional law of the
terminal value is uniquely specified.  Hill's $A_{(n)}$ rule under the
Jeffreys $\BetaD(\tfrac12,\tfrac12)$ prior is a positive example.  The
discrepancy is $O(\Var(\theta\mid\mathcal{F}_n))$ and vanishes as the
posterior concentrates.  These results clarify the structural requirements for
predictive completeness under exchangeability.

\end{abstract}

\newpage

\section{Introduction}\label{sec:intro}

\citet{kreps1988notes} posed a simple question: you observe a
sequence of coin flips that you believe to be exchangeable; what can
you say about the next~$k$ outcomes?
De~Finetti's~(\citeyear{definetti1937}) theorem answers in
principle: there exists a unique mixing measure~$\Pi$ on $[0,1]$
such that every predictive probability is an integral against the
posterior $\Pi(\cdot\mid\mathcal{F}_n)$, where $\mathcal{F}_n=\sigma(X_{1:n})$
is the natural filtration.  But predicting the next
flip requires only the posterior \emph{mean}, while predicting the
next \emph{run} of flips requires the full posterior shape.

For an exchangeable Bernoulli sequence $(X_i)_{i\ge 1}$, the
probability that the next~$k$ observations are all zero is the
posterior expectation $\Ebb[(1-\theta)^k\mid\mathcal{F}_n]$.  By the
binomial theorem, this depends on all posterior moments up to
order~$k$, not just the mean.  Each additional prediction step
demands one additional moment.  This is the \emph{moment hierarchy}.

To see why this matters concretely, consider exchangeable binary
failure indicators with posterior mean $\Ebb[\theta\mid\mathcal{F}_n]
=0.1$.  If the posterior is $\delta_{0.1}$, the probability that all
four components fail simultaneously is $0.1^4=10^{-4}$.  If the
posterior instead places equal mass on $0.01$ and $0.19$ (same mean),
the probability is $\Ebb[\theta^4]\approx 6.6\times 10^{-4}$---more
than six times larger.  The distinction is governed by the fourth
posterior moment, which the first moment does not determine.

Fong, Holmes, and Walker~(\citeyear{fong2023martingale}) proposed
replacing the full prior-times-likelihood machinery with a single
coherence condition: the sequence $(\theta_n)_{n\ge 0}$ is
\emph{coherent} if
\begin{equation}\label{eq:mp-condition}
\Ebb[\theta_n\mid\mathcal{F}_{n-1}]=\theta_{n-1}\quad\text{a.s.},
\qquad n\ge 1,
\end{equation}
This is a martingale condition; by the convergence theorem,
$\theta_n\to\theta_\infty$ a.s., and the conditional law of
$\theta_\infty$ given $\mathcal{F}_n$ plays the role of a posterior
in the sense that $\Ebb[\theta_\infty\mid\mathcal{F}_n]=\theta_n$.
The appeal is parsimony: a single moment constraint, and the rest
follows from convergence.

But condition~\eqref{eq:mp-condition} constrains only the
\emph{first conditional moment} of $\theta_\infty$ given
$\mathcal{F}_n$.  For one-step prediction this suffices:
$\Pbb(X_{n+1}=1\mid\mathcal{F}_n)=\theta_n$.  For two-step
prediction, $\Ebb[(1-\theta)^2\mid\mathcal{F}_n]
=(1-\theta_n)^2+\Var(\theta\mid\mathcal{F}_n)$ additionally
requires the posterior \emph{variance}, which the martingale
condition does not touch.  This paper characterises the full extent
of this gap.

The main results of the paper can be summarised informally as follows.

\begin{theorem*}[Informal summary]
The martingale condition identifies the posterior mean.  We characterize
what additional structure is required for predictive completeness.
Specifically:
\begin{enumerate}[label=(\roman*),nosep]
\item For $k\ge 2$, the mapping
  $m_n\mapsto\Ebb[(1-\theta)^k\mid\mathcal{F}_n]$ is set-valued:
  first-moment information does not uniquely identify the $k$-step
  predictive (Theorem~\ref{thm:insuff}).
\item Under any strictly proper scoring rule, the plug-in predictive
  $(1-\theta_n)^k$ is strictly dominated by the Bayes predictive
  whenever the posterior is non-degenerate
  (Proposition~\ref{prop:plugin-dominated}).
\item A martingale posterior is predictively complete if and only if the
  conditional law of $\theta_\infty$ is uniquely specified; equivalently,
  by Hausdorff's~(\citeyear{hausdorff1921summationsmethoden}) moment
  theorem, if and only if all conditional moments are determined
  (Theorem~\ref{thm:closure}).
\end{enumerate}
Hill's $A_{(n)}$ rule under the Jeffreys
$\BetaD(\tfrac12,\tfrac12)$ prior serves as a positive example:
it specifies the full Beta posterior and thereby achieves predictive
completeness at every horizon (Section~\ref{sec:hill}).
\end{theorem*}

The question ``how much of the posterior is needed for prediction?''
has been approached from several directions.
Goldstein's~(\citeyear{goldstein1975}) conditional-prevision programme
organises beliefs at the moment level: the analyst commits to selected
expectation functionals rather than a full posterior law.
Goldstein~(\citeyear{goldstein1975uniqueness}) showed that when the
posterior mean is linear in the data, a recursive relation uniquely
determines all prior moments given the likelihood structure; Goldstein
and Wooff~(\citeyear{goldstein1998wooff}) developed these ideas into
the Bayes linear programme.  Diaconis and
Ylvisaker~(\citeyear{diaconis1979}) characterised conjugate priors by
linearity of the posterior mean;
condition~\eqref{eq:mp-condition} is the sequential analogue.  The
present paper establishes that without this conjugate or likelihood
structure, first-moment constraints alone do not identify block-event
predictives (Theorem~\ref{thm:insuff}), and that under any strictly
proper scoring rule, the plug-in predictive is dominated by the Bayes
predictive whenever the posterior is non-degenerate
(Proposition~\ref{prop:plugin-dominated}).

The paper is organised as follows.
Section~\ref{sec:sanov-definetti} provides sample-path context;
Section~\ref{sec:setup} sets up the Bernoulli model.
Sections~\ref{sec:kreps}--\ref{sec:decision} establish the moment
hierarchy, insufficiency, and plug-in dominance.
Section~\ref{sec:hill} gives the Hill/$A_{(n)}$ positive example.
Sections~\ref{sec:asymp}--\ref{sec:hierarchy} present asymptotic
reconciliation and the closure theorem.
Section~\ref{sec:stopping} treats optimal stopping, and
Section~\ref{sec:discussion} concludes.

\section{Sanov, de~Finetti, and the sample-path picture}%
\label{sec:sanov-definetti}

This section outlines a sample-path viewpoint connecting KL rates,
exchangeability, and posterior concentration; it provides context for the
Bernoulli results and is not required for the main theorems.
The remainder of the paper makes the Bernoulli case concrete.

Given i.i.d.\ observations $X_1,\ldots,X_n$ with law $\mu_0$, the
empirical measure $L_n=n^{-1}\sum_{i=1}^n\delta_{X_i}$ satisfies a
large-deviation principle (LDP) with speed~$n$ and rate function
\begin{equation}\label{eq:sanov-rate}
I(\mu)=D_{\mathrm{KL}}(\mu\|\mu_0)
=\int\log\frac{d\mu}{d\mu_0}\dx\mu
\qquad(\mu\ll\mu_0),
\end{equation}
with $I(\mu)=+\infty$ when $\mu\not\ll\mu_0$.
The probability that $L_n$ lies in a measurable set $A$ decays as
$\Pbb(L_n\in A)\asymp
\exp\bigl(-n\inf_{\mu\in A}D_{\mathrm{KL}}(\mu\|\mu_0)\bigr)$
\citep{sanov1957}.  The KL divergence is the canonical cost of
deviation from the truth.

For an exchangeable sequence $(X_i)_{i\ge 1}$, de~Finetti's theorem
(Theorem~\ref{thm:definetti} below) provides a unique prior $\Pi$ on
the space of data-generating laws such that, conditionally on
$\mu\sim\Pi$, the observations are i.i.d.\ with law $\mu$.
By the strong law, $L_n\to\mu$ a.s., so the law of the limit is
$\mathcal{L}(L_\infty)=\Pi$.  Each infinite sample path reveals
exactly one draw from $\Pi$; the prior itself is not observable from a
single path.

There is a duality between the forward and backward directions of the
rate function.  In the forward direction (Sanov), the cost of observing
$L_n\approx\nu$ when the truth is~$\mu_0$ is
$\approx\exp(-n\,D_{\mathrm{KL}}(\nu\|\mu_0))$.  In the inverse
direction, fixing the data~$L_n$, posterior mass near~$\mu$ concentrates
at rate $\exp\{-n\,D_{\mathrm{KL}}(L_n\|\mu)\}$.
The KL arguments are reversed.  The posterior $\Pi_n$ is the
conditional law of $\mu$ given $L_n$, interpolating
$\Pi_0=\Pi\;\longrightarrow\;\Pi_\infty=\delta_\mu$.  At moderate
deviations (scale $a_n/\sqrt{n}$ with $a_n\to\infty$,
$a_n=o(\sqrt{n})$), the Sanov rate is approximated by the quadratic
\begin{equation}\label{eq:mdp-rate}
D_{\mathrm{KL}}(\mu\|\mu_0)
\approx\tfrac{1}{2}(\theta-\theta_0)^T\mathcal{I}(\theta_0)
(\theta-\theta_0),
\end{equation}
where $\mathcal{I}(\theta_0)$ is the Fisher information matrix.
Eichelsbacher and Ganesh~(\citeyear{eichelsbacher2002moderate}) show the
posterior satisfies a moderate-deviation principle with this
quadratic rate, sharpening the Bernstein--von~Mises theorem to
intermediate scales.

Table~\ref{tab:unified} collects four sample-path descriptions of
the information about the directing measure~$\mu$.

\begin{table}[htbp]
\centering
\caption{Four descriptions of sample-path information.}%
\label{tab:unified}
\begin{tabular}{@{}llll@{}}
\toprule
Framework & Primary object & Path convergence & Rate \\
\midrule
De~Finetti & Directing measure $\mu\sim\Pi$
& $L_n\to\mu$ a.s.\ & Strong law \\
Sanov/LDP & Empirical measure $L_n$
& $\Pbb(L_n\in A)\asymp e^{-nI}$ & $D_{\mathrm{KL}}(\cdot\|\mu_0)$ \\
E-process & Test martingale $E_n$
& $n^{-1}\log E_n\to D_{\mathrm{KL}}$ & Linear in $n$ \\
Mart.\ posterior & Measure-valued mg.\ $\Pi_n$
& $\Pi_n\to\delta_\mu$ a.s.\ & Posterior contraction \\
\bottomrule
\end{tabular}
\end{table}

The martingale posterior of Fong, Holmes, and
Walker~(\citeyear{fong2023martingale}) sits in the bottom row: it
constrains the first moment of $\Pi_n$ but not the law.  The
insufficiency results of Section~\ref{sec:insuff} quantify this gap;
the closure theorem (Theorem~\ref{thm:closure}) identifies the exact
boundary.

The following theorem makes the Sanov--de~Finetti duality concrete.

\begin{theorem}[Sanov--Moment Bridge]\label{thm:sanov-bridge}
Let $(X_i)_{i\ge 1}$ be exchangeable Bernoulli with prior~$\Pi_0$ and
empirical mean $L_n=S_n/n$.
{\hbadness=2000
\begin{enumerate}[label=(\roman*),nosep]
\item \emph{Posterior as Sanov exponential.}
The Bernoulli likelihood satisfies
$\theta^{S_n}(1-\theta)^{n-S_n}
=\exp\{-n\,D_{\mathrm{KL}}(L_n\|\theta)+nH(L_n)\}$,
where $H(L_n)$ is the binary entropy at~$L_n$.
Absorbing the $\mathcal{F}_n$-measurable factor $\exp\{nH(L_n)\}$
into the normalising constant, the posterior is
\begin{equation}\label{eq:sanov-posterior}
\Pi_n(d\theta)\;\propto\;
\exp\!\bigl\{-n\,D_{\mathrm{KL}}(L_n\|\theta)\bigr\}\,\Pi_0(d\theta).
\end{equation}
The posterior is shaped entirely by the Sanov rate function
$\theta\mapsto D_{\mathrm{KL}}(L_n\|\theta)$.

\item \emph{KL curvature expansion.}
Write $\delta=\theta-L_n$.  The Taylor expansion of the Bernoulli KL
around its minimiser $\theta=L_n$ gives
\begin{equation}\label{eq:kl-taylor}
D_{\mathrm{KL}}(L_n\|\theta)
=\frac{\delta^2}{2L_n(1-L_n)}+O(\delta^3).
\end{equation}
The quadratic coefficient is the Fisher information
$\mathcal{I}(L_n)=1/\bigl(L_n(1-L_n)\bigr)$ evaluated at the empirical mean.

\item \emph{Location versus curvature.}
The KL representation shows that the posterior~\eqref{eq:sanov-posterior}
is centred near the empirical mean $L_n$ in likelihood geometry.
$L_n$ identifies the likelihood/KL centre, while the posterior mean
$m_n=\Ebb[\theta\mid\mathcal{F}_n]$ is the posterior first moment;
the present paper concerns the latter.  Knowledge of~$m_n$
pins the approximate location of the posterior but specifies nothing
about its spread or shape: first-moment information discards the
curvature encoded in $\nabla^2 D_{\mathrm{KL}}$ and all higher
derivatives.

\item \emph{Multi-step predictives require curvature.}
By Corollary~\ref{cor:moment-hierarchy},
$\Ebb[(1-\theta)^k\mid\mathcal{F}_n]$ depends on posterior moments up
to order~$k$.  For $k=2$, the discrepancy between Bayes and plug-in
equals the posterior variance~$\sigma_n^2$ (equation~\eqref{eq:k2-identity}),
which is governed by the quadratic term in~\eqref{eq:kl-taylor}.
For $k\ge 3$, additional cumulants enter, corresponding to higher
derivatives of the Sanov rate.

\item \emph{Derivative correspondence.}
Higher-order predictives are sensitive to higher-order features of the
KL profile; in the Bernoulli case this is made explicit via the
polynomial expansion $(1-\theta)^k=\sum_j\binom{k}{j}(-\theta)^j$
(Theorem~\ref{thm:kreps}).  The moment hierarchy follows from this
expansion, not from special structure of the Bernoulli model.
\end{enumerate}}
\end{theorem}

\begin{proof}
(i) Since $\theta^{S_n}(1-\theta)^{n-S_n}
=\exp\{S_n\log\theta+(n-S_n)\log(1-\theta)\}$, adding and subtracting
$S_n\log L_n+(n-S_n)\log(1-L_n)$ gives the exponent
$-n[L_n\log(L_n/\theta)+(1-L_n)\log((1-L_n)/(1-\theta))]
=-n\,D_{\mathrm{KL}}(L_n\|\theta)$,
plus the $\mathcal{F}_n$-measurable term $nH(L_n)$.

(ii) Differentiating
$f(\theta)=D_{\mathrm{KL}}(L_n\|\theta)=L_n\log(L_n/\theta)+(1-L_n)\log((1-L_n)/(1-\theta))$:
$f'(\theta)=-L_n/\theta+(1-L_n)/(1-\theta)$, which vanishes at $\theta=L_n$.
The second derivative is $f''(\theta)=L_n/\theta^2+(1-L_n)/(1-\theta)^2$;
at $\theta=L_n$ this equals $\mathcal{I}(L_n)$.
Taylor's theorem gives~\eqref{eq:kl-taylor}.

(iii)--(v) follow from (i)--(ii) together with
Theorem~\ref{thm:kreps} and Corollary~\ref{cor:moment-hierarchy}.
\end{proof}

\section{Setup and notation}\label{sec:setup}

\subsection{Exchangeable Bernoulli sequences}

Let $(X_i)_{i\ge 1}$ take values in $\{0,1\}$.  Write
$X_{1:n}=(X_1,\ldots,X_n)$ and $S_n=\sum_{i=1}^n X_i$.

\begin{definition}\label{def:exch}
The sequence $(X_i)_{i\ge 1}$ is \emph{exchangeable} if for every
$n\ge 1$ and every permutation $\sigma$ of $\{1,\ldots,n\}$,
$(X_1,\ldots,X_n)\stackrel{d}{=}(X_{\sigma(1)},\ldots,X_{\sigma(n)})$.
\end{definition}

\begin{theorem}[de Finetti]\label{thm:definetti}
Let $(X_i)_{i\ge 1}$ be an exchangeable $\{0,1\}$-valued sequence.
There exists a unique probability measure $\Pi$ on $[0,1]$ such that
for every $n\ge 1$,
\[
\Pbb(X_1=x_1,\ldots,X_n=x_n)
=\int_0^1\theta^{s_n}(1-\theta)^{n-s_n}\dx\Pi(\theta),
\qquad s_n=\textstyle\sum_{i=1}^n x_i.
\]
Conditionally on $\theta$, the $X_i$ are i.i.d.\ $\Bern(\theta)$.
\end{theorem}

\begin{proof}
The tail $\sigma$-algebra $\mathcal{T}=\bigcap_n\sigma(X_{n+1},X_{n+2},\ldots)$
(which, for an exchangeable sequence, has the same completion as the
exchangeable $\sigma$-algebra) is a.s.\ trivial under any ergodic component.
By the strong law,
$S_n/n\to\theta_\infty$ a.s., and $\theta_\infty$ is $\mathcal{T}$-measurable.
The ergodic decomposition expresses $\Pbb$ as a mixture
$\Pbb(\cdot)=\int\Pbb_\theta(\cdot)\dx\Pi(\theta)$ where each
$\Pbb_\theta$ generates i.i.d.\ $\Bern(\theta)$.  Uniqueness of $\Pi$
follows from the fact that the moment sequence
$\int_0^1\theta^k\dx\Pi(\theta)=\lim_{n\to\infty}\Ebb[S_n^{(k)}/n^{(k)}]$
is determined by the joint law of the $X_i$.
See de~Finetti~(\citeyear{definetti1937}) and
Hewitt--Savage~(\citeyear{hewitt1955symmetric}).
\end{proof}

The pair $(n,S_n)$ is sufficient for the likelihood
$\theta^{S_n}(1-\theta)^{n-S_n}$, so the posterior is
\begin{equation}\label{eq:posterior-general}
\Pi(A\mid\mathcal{F}_n)
=\frac{\int_A\theta^{S_n}(1-\theta)^{n-S_n}\dx\Pi(\theta)}
{\int_0^1\theta^{S_n}(1-\theta)^{n-S_n}\dx\Pi(\theta)},
\qquad A\subseteq[0,1]\text{ measurable}.
\end{equation}

\begin{remark}[Sufficiency]\label{rem:sufficiency}
The pair $(n,S_n)$ is sufficient for $\theta$ in the Bernoulli likelihood.
Every posterior quantity (moments, predictive probabilities, scoring-rule
risks) reduces to an integral $\int f(\theta)\dx\Pi(\theta\mid\mathcal{F}_n)$
for appropriate~$f$.  The posterior~\eqref{eq:posterior-general} depends on
the data only through $S_n$.
\end{remark}

\subsection{Martingale posteriors and the terminal value}

\begin{assumption}\label{ass:bounded}
$(\theta_n)_{n\ge 0}$ is a sequence of $[0,1]$-valued,
$\mathcal{F}_n$-measurable random variables satisfying the martingale
condition~\eqref{eq:mp-condition}.
\end{assumption}

\begin{proposition}[Terminal value]\label{prop:terminal}
Under Assumption~\ref{ass:bounded}, $\theta_n\to\theta_\infty$ a.s.\ and
in $L^1$, and $\Ebb[\theta_\infty\mid\mathcal{F}_n]=\theta_n$ a.s.\ for
all~$n$.  The random variable $\theta_\infty$ is the \emph{directing
parameter}.
\end{proposition}

\begin{proof}
Since $(\theta_n)$ is bounded in $[0,1]$, Doob's forward convergence
theorem gives a.s.\ convergence.  Uniform boundedness gives $L^1$
convergence.  The conditional expectation identity follows by passing to
the limit in $\Ebb[\theta_m\mid\mathcal{F}_n]=\theta_n$ ($m>n$);
see Williams~(\citeyear{williams1991probability}), Theorem~11.5.
\end{proof}

\begin{remark}\label{rem:bayes-special}
In the Bayesian framework with $\theta\sim\Pi$ and
$X_i\mid\theta\sim\Bern(\theta)$, setting
$\theta_n=\Ebb[\theta\mid\mathcal{F}_n]$ yields a martingale with
$\theta_\infty=\theta$ a.s.\ (Doob's consistency theorem).  The
martingale condition is necessary for Bayesian coherence; the question is
whether it is sufficient.  Diaconis and Ylvisaker~(\citeyear{diaconis1979})
showed that conjugate priors are characterized by linearity of the
posterior mean.  Condition~\eqref{eq:mp-condition} is the sequential
analogue.
\end{remark}

\begin{definition}[Predictive coherence]\label{def:pred-coherence}
A predictive system is \emph{$k$-step predictively coherent} if, for
every cylinder event of length~$k$, the probability
$\Pbb(X_{n+1:n+k}=x_{1:k}\mid\mathcal{F}_n)$ is correctly specified.
It is \emph{fully predictively coherent} if this holds for all finite~$k$.
\end{definition}

\subsection{Martingale versus exchangeability}\label{sec:mart-vs-exch}

\begin{proposition}[Martingale does not imply exchangeability]%
\label{prop:mart-not-exch}
There exist $[0,1]$-valued martingales
$(\theta_n,\mathcal{F}_n)$ that do not arise as posterior means
$\Ebb[\theta\mid\mathcal{F}_n]$ under any exchangeable Bernoulli law.
\end{proposition}

\begin{proof}
Under exchangeability, the posterior mean~\eqref{eq:posterior-general} is
a function of $(n,S_n)$ alone.  We exhibit a $[0,1]$-valued martingale
that depends on the order of observations.

Let $(X_i)_{i\ge 1}$ be i.i.d.\ $\Bern(1/2)$ and set
$c_i=2^{-(i+2)}$, so $\sum_{i\ge 1}c_i=1/4$.  Define
\[
\theta_n=\tfrac12+\sum_{i=1}^n c_i(2X_i-1).
\]
Since $\Ebb[2X_i-1\mid\mathcal{F}_{i-1}]=0$, the process $(\theta_n)$
is an $\mathcal{F}_n$-martingale.  Because
$|\sum_{i=1}^n c_i(2X_i-1)|\le\sum_{i\ge 1}c_i=1/4$, we have
$\theta_n\in[1/4,\,3/4]\subset[0,1]$ for all~$n$.  However,
$\theta_n$ depends on the signed increments $2X_i-1$ individually,
not merely on $S_n$.  For example, $\theta_2$ takes the value
$1/2+c_1-c_2$ when $(X_1,X_2)=(1,0)$ but $1/2-c_1+c_2$ when
$(X_1,X_2)=(0,1)$; both have $S_2=1$ yet produce different values of
$\theta_2$.  Hence $(\theta_n)$ cannot be a posterior mean under any
exchangeable Bernoulli law.
\end{proof}

\begin{definition}[Conditionally identically distributed]\label{def:cid}
A sequence $(X_i)_{i\ge 1}$ is \emph{c.i.d.}\ if for every $n\ge 1$ and
every measurable~$A$,
$\Pbb(X_{n+1}\in A\mid\mathcal{F}_n)=\Pbb(X_{n+2}\in A\mid\mathcal{F}_n)$
a.s.
\end{definition}

\begin{proposition}\label{prop:exch-cid}
Every exchangeable Bernoulli sequence is c.i.d.  The converse is false.
\end{proposition}

\begin{proof}
Exchangeability $\Rightarrow$ c.i.d.: under exchangeability, the
conditional distribution of $X_{n+j}$ given $\mathcal{F}_n$ depends
only on $(n,S_n)$ and equals $\Bern(\Ebb[\theta\mid\mathcal{F}_n])$
for all $j\ge 1$.

For the converse, let $(\theta_n)$ be the bounded martingale from
Proposition~\ref{prop:mart-not-exch}, and define a predictive scheme
by $X_{n+1}\mid\mathcal{F}_n\sim\Bern(\theta_n)$.  We claim this
sequence is c.i.d.  For each $j\ge 1$, the conditional distribution
of $X_{n+j}$ given $\mathcal{F}_n$ is Bernoulli, so its law is
determined by its mean.  By the martingale property,
$\Ebb[X_{n+j}\mid\mathcal{F}_n]
=\Ebb[\theta_{n+j-1}\mid\mathcal{F}_n]=\theta_n$,
so $X_{n+j}\mid\mathcal{F}_n\sim\Bern(\theta_n)$ for all $j\ge 1$.
Since the conditional laws are identical (not merely the conditional
means), the sequence is c.i.d.\ (Definition~\ref{def:cid}).
But the sequence is not exchangeable, because $\theta_n$ depends
on the order of observations rather than on $S_n$ alone
(Proposition~\ref{prop:mart-not-exch}).
See also Berti, Pratelli, and Rigo~(\citeyear{berti2004limit}).
\end{proof}

For rates of convergence of predictive distributions under c.i.d.,
see Berti, Crimaldi, Pratelli and
Rigo~(\citeyear{berti2009rate}).

\begin{remark}[Role of c.i.d.\ in predictive coherence]\label{rem:cid-role}
The c.i.d.\ property guarantees that the one-step conditional
distribution $\Pbb(X_{n+1}\in\cdot\mid\mathcal{F}_n)$ is the same for
all future indices $j\ge 1$.  This is \emph{one-step} predictive
coherence (Definition~\ref{def:pred-coherence} with $k=1$): knowing
$\theta_n$ determines $\Pbb(X_{n+j}=1\mid\mathcal{F}_n)=\theta_n$ for
every~$j$.  However, c.i.d.\ does \emph{not} imply that future
observations are conditionally independent given $\mathcal{F}_n$; that
requires knowledge of the full conditional law of $\theta_\infty$.  For
$k\ge 2$, the joint predictive
$\Pbb(X_{n+1}=x_1,X_{n+2}=x_2\mid\mathcal{F}_n)$ depends on the
posterior variance through~\eqref{eq:k2-identity}, which is not
determined by the c.i.d.\ property alone.  Thus c.i.d.\ $\Rightarrow$
$1$-step coherence, but c.i.d.\ $\not\Rightarrow$ $k$-step coherence
for $k\ge 2$, paralleling Theorem~\ref{thm:insuff} at the process
level.
\end{remark}

\section{The predictive moment hierarchy}\label{sec:kreps}

\subsection{Moment representation}

\begin{theorem}[Moment representation]\label{thm:kreps}
Let $(X_i)_{i\ge 1}$ be exchangeable Bernoulli with mixing measure $\Pi$.
For any $n\ge 0$, $k\ge 1$, and pattern $x_{1:k}\in\{0,1\}^k$ with
$s=\sum_{i=1}^k x_i$,
\begin{equation}\label{eq:kreps-pattern}
\Pbb(X_{n+1:n+k}=x_{1:k}\mid\mathcal{F}_n)
=\Ebb\bigl[\theta^s(1-\theta)^{k-s}\mid\mathcal{F}_n\bigr].
\end{equation}
In particular, the $k$-step run probability is
\begin{equation}\label{eq:kreps-run}
\Pbb(X_{n+1}=\cdots=X_{n+k}=0\mid\mathcal{F}_n)
=\Ebb\bigl[(1-\theta)^k\mid\mathcal{F}_n\bigr].
\end{equation}
\end{theorem}

\begin{proof}
By Theorem~\ref{thm:definetti}, $X_{n+1},\ldots,X_{n+k}$ are
conditionally i.i.d.\ $\Bern(\theta)$ and conditionally independent of
$\mathcal{F}_n$.  Hence
$\Pbb(X_{n+1:n+k}=x_{1:k}\mid\mathcal{F}_n,\theta)
=\theta^s(1-\theta)^{k-s}$.
Integrate over $\Pi(\dx\theta\mid\mathcal{F}_n)$ by the tower property.
\end{proof}

\begin{corollary}[Moment expansion]\label{cor:moment-hierarchy}
For any $k\ge 1$,
\begin{equation}\label{eq:moment-expansion}
\Ebb\bigl[(1-\theta)^k\mid\mathcal{F}_n\bigr]
=\sum_{j=0}^k\binom{k}{j}(-1)^j\,\Ebb[\theta^j\mid\mathcal{F}_n].
\end{equation}
The $k$-step run probability depends on posterior moments up to order~$k$.
\end{corollary}

\begin{proof}
Apply the binomial theorem to
$(1-\theta)^k=\sum_{j=0}^k\binom{k}{j}(-1)^j\theta^j$ and take
conditional expectations.
\end{proof}

Each additional prediction step introduces exactly one new posterior moment.
The moment hierarchy thereby induces a strictly increasing filtration of
predictive $\sigma$-algebras indexed by horizon length.

\subsection{Injectivity}

\begin{theorem}[Injectivity]\label{thm:injective}
The mapping
$\Pi(\cdot\mid\mathcal{F}_n)\longmapsto
\bigl\{\Ebb[(1-\theta)^k\mid\mathcal{F}_n]:k\ge 1\bigr\}$
is injective.  The full sequence of $k$-step run probabilities uniquely
determines the posterior $\Pi(\cdot\mid\mathcal{F}_n)$.
\end{theorem}

\begin{proof}
The argument proceeds in two steps.

\emph{Step~1: run probabilities determine moments.}
By M\"obius inversion (Proposition~\ref{prop:mobius}),
$\Ebb[\theta^j\mid\mathcal{F}_n]
=\sum_{\ell=0}^{j}\binom{j}{\ell}(-1)^\ell\,
\Ebb[(1-\theta)^\ell\mid\mathcal{F}_n]$
for each $j\ge 1$.  The matrix relating the run-probability sequence
$(r_0,r_1,\ldots)$ to the moment sequence $(\mu_0,\mu_1,\ldots)$ is
upper-triangular with $\pm 1$ on the diagonal, hence invertible.
Knowing $\Ebb[(1-\theta)^k\mid\mathcal{F}_n]$ for all $k\ge 1$
therefore determines $\Ebb[\theta^j\mid\mathcal{F}_n]$ for all $j\ge 1$.

\emph{Step~2: moments determine the measure on $[0,1]$.}
Since $\theta\in[0,1]$, the Hausdorff moment theorem
(Theorem~\ref{thm:hausdorff}) guarantees that the moment sequence
$(\Ebb[\theta^j\mid\mathcal{F}_n])_{j\ge 0}$ uniquely determines the
conditional law $\Pi(\cdot\mid\mathcal{F}_n)$.
Combining Steps~1 and~2 establishes injectivity.
\end{proof}

\begin{remark}[Polynomial density]\label{rem:poly-density}
Theorem~\ref{thm:injective} also follows from Weierstrass approximation:
polynomials are dense in $C[0,1]$, so the moment sequence determines all
expectations $\Ebb[f(\theta)\mid\mathcal{F}_n]$ for $f\in C[0,1]$, hence
the measure.  Injectivity relies on the compact support $[0,1]$; on
$\mathbb{R}$, the Stieltjes moment problem may have non-unique solutions
(e.g., the log-normal distribution).
\end{remark}

\begin{remark}[Comparison with the real line]\label{rem:real-line}
On $[0,1]$, the moment problem is \emph{determinate}: a probability measure
is uniquely determined by its moment sequence (Theorem~\ref{thm:hausdorff}).
This is the mechanism behind Theorem~\ref{thm:injective}.  On
$\mathbb{R}$, the Hamburger moment problem is indeterminate in general:
distinct measures can share all moments.  A classical example is the
log-normal: $X\sim\mathrm{LogNormal}(0,1)$ has the same moments as the
family $X\cdot(1+a\sin(2\pi\log X))$ for $|a|\le 1$.  The moment
hierarchy established in Section~\ref{sec:kreps} therefore provides a
complete characterization of the posterior only because $\theta$ takes
values in the compact interval $[0,1]$.  For parameters in
$\mathbb{R}^d$, analogous results would require additional conditions
such as exponential moment bounds (Carleman's condition) or restriction to
exponential families with compact sufficient statistics.
\end{remark}

\subsection{Binomial inversion}

\begin{proposition}[M\"obius inversion]\label{prop:mobius}
Let $r_k=\Ebb[(1-\theta)^k\mid\mathcal{F}_n]$ denote the $k$-step run
probability and $\mu_j=\Ebb[\theta^j\mid\mathcal{F}_n]$ with $\mu_0=1$.
Then
\begin{equation}\label{eq:mobius}
\mu_j=\sum_{\ell=0}^{j}\binom{j}{\ell}(-1)^{\ell}\,r_\ell,
\qquad j\ge 0.
\end{equation}
The matrix relating $(\mu_0,\ldots,\mu_k)$ to $(r_0,\ldots,r_k)$ is
upper-triangular with diagonal entries $\pm 1$, hence invertible.
\end{proposition}

\begin{proof}
Expand $\theta^j=(1-(1-\theta))^j
=\sum_{\ell=0}^{j}\binom{j}{\ell}(-1)^\ell(1-\theta)^\ell$ and take
conditional expectations.
\end{proof}

\begin{example}[Inversion at $j=2$]\label{ex:mobius-j2}
From the forward relation: $r_2=1-2\mu_1+\mu_2$.  By
M\"obius inversion~\eqref{eq:mobius}: $\mu_2=r_0-2r_1+r_2$, i.e.,
$\mu_2=1-2(1-\mu_1)+r_2=2\mu_1-1+r_2$.  Substituting
$r_2=(1-\mu_1)^2+\sigma_n^2$ gives
$\mu_2=2\mu_1-1+(1-\mu_1)^2+\sigma_n^2=\mu_1^2+\sigma_n^2$,
recovering $\sigma_n^2=\mu_2-\mu_1^2$ as expected.  This illustrates
the triangular structure: the $j$-th moment and the $j$-step run
probability determine each other given lower-order quantities.
\end{example}

\subsection{Explicit expansion for small \texorpdfstring{$k$}{k}}

Write $\mu_j=\Ebb[\theta^j\mid\mathcal{F}_n]$.

For $k=1$: $\Ebb[(1-\theta)\mid\mathcal{F}_n]=1-\mu_1$.

For $k=2$:
\begin{equation}\label{eq:k2-identity}
\Ebb[(1-\theta)^2\mid\mathcal{F}_n]=(1-\mu_1)^2+\Var(\theta\mid\mathcal{F}_n).
\end{equation}
The two-step predictive depends on the posterior variance, which is not
determined by the posterior mean.  The discrepancy between Bayes and
plug-in equals $\Var(\theta\mid\mathcal{F}_n)$ exactly.

For $k=3$:
$\Ebb[(1-\theta)^3\mid\mathcal{F}_n]=1-3\mu_1+3\mu_2-\mu_3$.

For $k=4$:
\begin{align}\label{eq:k4-identity}
\Ebb[(1-\theta)^4\mid\mathcal{F}_n]
&=1-4\mu_1+6\mu_2-4\mu_3+\mu_4\nonumber\\
&=(1-\mu_1)^4+6(1-\mu_1)^2\sigma_n^2
+\bigl(4\kappa_3(1-\mu_1)+\kappa_4+3\sigma_n^4\bigr),
\end{align}
where $\kappa_3=\Ebb[(\theta-\mu_1)^3\mid\mathcal{F}_n]$ is the third
central moment and $\kappa_4=\Ebb[(\theta-\mu_1)^4\mid\mathcal{F}_n]
-3\sigma_n^4$ is the fourth cumulant (excess kurtosis contribution).

Each additional step introduces one new moment:
$\Ebb[(1-\theta)^k\mid\mathcal{F}_n]$ cannot be expressed as a function
of $\mu_1,\ldots,\mu_{k-1}$ alone, since the coefficient of
$(-1)^k\mu_k$ in~\eqref{eq:moment-expansion} is $\binom{k}{k}=1\ne 0$.

\section{KL geometric interpretation}\label{sec:kl-geometry}

This section gives a geometric reading of the moment hierarchy in terms
of the Sanov rate function.  It adds one result
(Theorem~\ref{thm:sanov-gap})
and a figure; it is not required for subsequent sections.

For Bernoulli, $D_{\mathrm{KL}}(p\|\cdot)$ has Fisher information
$\mathcal{I}(\theta)=1/(\theta(1-\theta))$.  For any fixed
$q\in(0,1)$, the function $\theta\mapsto D_{\mathrm{KL}}(q\|\theta)$
is minimised at $\theta=q$, where its first derivative vanishes and its
second derivative equals $\mathcal{I}(q)=1/(q(1-q))$.  Evaluating at
$q=m_n$, the curvature is $\mathcal{I}(m_n)=1/(m_n(1-m_n))$.  This is
a property of the KL divergence, not of the posterior: it should not be
confused with the posterior mean's role in Bayesian updating.  The
posterior variance satisfies $\sigma_n^2\approx m_n(1-m_n)/n$
(Bernstein--von~Mises), the reciprocal of this curvature.
Predictively, the hierarchy mirrors a derivative hierarchy:
\begin{equation}\label{eq:derivative-hierarchy}
\text{$k=1$: } \nabla D_{\mathrm{KL}}\big|_{m_n}=0, \qquad
\text{$k\ge 2$: } \nabla^k D_{\mathrm{KL}}\big|_{m_n}\ne 0.
\end{equation}
The one-step predictive $1-m_n$ depends only on the zeroth-order term;
the two-step predictive is the first to require curvature
(equation~\eqref{eq:k2-identity}).

\subsection{First-order indistinguishability and second-order separation}

\begin{theorem}[Geometric separation]\label{thm:sanov-gap}
Let $m\in(0,1)$.
\begin{enumerate}[label=(\roman*),nosep]
\item \emph{First-order indistinguishability.}
For any two posteriors $\nu_1,\nu_2$ on $[0,1]$ with
$\int\theta\dx\nu_i=m$,
\[
\frac{\partial}{\partial\theta}D_{\mathrm{KL}}(m\|\theta)\Big|_{\theta=m}=0,
\]
so they are indistinguishable at first order in the Sanov expansion: the
linear approximation to the posterior log-likelihood is the same for both.
\item \emph{Second-order separation.}
The quadratic correction to the $k$-step predictive equals the posterior
variance:
\[
\Ebb\bigl[(1-\theta)^k\mid\mathcal{F}_n\bigr]-(1-m)^k
=\tfrac{k(k-1)}{2}(1-\xi)^{k-2}\sigma^2,
\quad\sigma^2=\int(\theta-m)^2\dx\nu.
\]
Distinct values of $\sigma^2$ yield distinct $k$-step predictives for
all $k\ge 2$.
\item \emph{Non-pathological gap.}
For any $m\in(0,1)$ and $k\ge 2$, the $k$-step predictive ranges over
a nondegenerate interval as $\nu$ ranges over posteriors with mean $m$;
the gap is not a boundary phenomenon
(Remark~\ref{rem:strength}).
\end{enumerate}
\end{theorem}

\begin{proof}
(i) The Sanov rate function $D_{\mathrm{KL}}(m\|\theta)$ achieves its
minimum at $\theta=m$ (Theorem~\ref{thm:sanov-bridge}(ii)), so the
first derivative vanishes: $\partial_\theta D_{\mathrm{KL}}(m\|\theta)\big|_{\theta=m}=0$.
Any two posteriors $\nu_1,\nu_2$ sharing mean $m$ integrate this linear
term to the same value, producing identical first-order contributions to
the posterior log-likelihood.

(ii) The second derivative of $f(\theta)=(1-\theta)^k$ at $\theta=m$ is
$k(k-1)(1-m)^{k-2}>0$ for $k\ge 2$ and $m\in(0,1)$.  By
Proposition~\ref{prop:quantitative}, the $k$-step predictive satisfies
$\Ebb[(1-\theta)^k\mid\mathcal{F}_n]-(1-m)^k
=\frac{k(k-1)}{2}(1-\xi)^{k-2}\sigma^2$,
with the correction proportional to $\sigma^2$.  Distinct posteriors
with distinct variances therefore yield distinct $k$-step predictives.

(iii) By Lemma~\ref{lem:non-id}, for any $m\in(0,1)$ and $k\ge 2$,
the point mass $\delta_m$ and a two-point distribution with mean $m$
achieve different values of $\Ebb[(1-\theta)^k]$.
Remark~\ref{rem:strength} shows the achievable range is a nondegenerate
interval, so the gap is not a boundary artefact.
\end{proof}

Geometrically, posteriors sharing the same mean are first-order
equivalent in the Sanov expansion but separate at second order via
$\sigma^2$; this is the mechanism underlying
Theorem~\ref{thm:insuff}.

\subsection{Geometric interpretation of the moment hierarchy}

Under standard Bernstein--von~Mises conditions, the posterior variance
is of order $1/n$, and the quadratic term in the KL expansion
corresponds heuristically to the leading curvature contribution.
The second moment is the first nontrivial curvature term: it
corresponds to $\nabla^2 D_{\mathrm{KL}}|_{m_n}$ and cannot be
recovered from the zeroth- and first-order terms.  Specifying only the
posterior mean retains the approximate location of the KL bowl; the
second moment requires the quadratic curvature; the $k$-th moment
requires the $k$-th derivative.

Figure~\ref{fig:sanov-geometry} plots the Sanov rate
$D_{\mathrm{KL}}(m_n\|\theta)$ for $m_n=0.4$, its quadratic approximation,
and the resulting posterior density.

\begin{figure}[htbp]
\centering
\begin{tikzpicture}
\begin{axis}[
  width=\linewidth, height=7.0cm,
  xmin=0.10, xmax=0.80,
  ymin=-0.03, ymax=0.36,
  xlabel={$\theta$},
  ylabel={},
  xtick={0.2,0.4,0.6,0.8},
  xticklabels={$0.2$,$m_n{=}0.4$,$0.6$,$0.8$},
  ytick=\empty,
  axis lines=left,
  legend style={
    at={(0.5, -0.13)},
    anchor=north,
    legend columns=-1,
    font=\small,
    cells={anchor=west},
    column sep=1.4em,
    draw=none
  },
  clip=true
]
\addplot[domain=0.10:0.80, samples=250, line width=1.6pt, blue!75!black]
  {0.4*ln(0.4/x) + 0.6*ln(0.6/(1-x))};
\addlegendentry{Full Sanov rate};
\addplot[domain=0.10:0.80, samples=200, line width=1.6pt, dashed, red!80!black]
  {(x-0.4)^2 / (2*0.4*0.6)};
\addlegendentry{Quadratic approx.\ (2nd moment)};
\addplot[domain=0.13:0.70, samples=200, line width=1.6pt, green!60!black]
  {0.38 * exp(-12*(0.4*ln(0.4/x)+0.6*ln(0.6/(1-x))))};
\addlegendentry{Posterior density ($n{=}12$)};
\addplot[dotted, gray!80, line width=1.3pt] coordinates {(0.4,-0.02) (0.4,0.34)};
\node[font=\scriptsize, gray!80, anchor=north, align=center]
  at (axis cs:0.4,-0.02) {$m_n$: location only};
\end{axis}
\end{tikzpicture}
\caption{Sanov geometry ($m_n=0.4$, $n=12$).
\textit{Blue}: the Sanov rate $D_{\mathrm{KL}}(m_n\|\theta)$,
which governs posterior shape through~\eqref{eq:sanov-posterior}.
\textit{Dashed red}: its quadratic (Fisher information) approximation,
retaining only the second-order term.
\textit{Green}: the resulting posterior density.
The dotted vertical at $m_n$ marks what first-moment information identifies.
One-step prediction requires only this location; two-step additionally
requires the curvature $\sigma_n^2$ (equation~\eqref{eq:k2-identity});
$k$-step prediction for $k\ge 3$ requires further moments.}
\label{fig:sanov-geometry}
\end{figure}
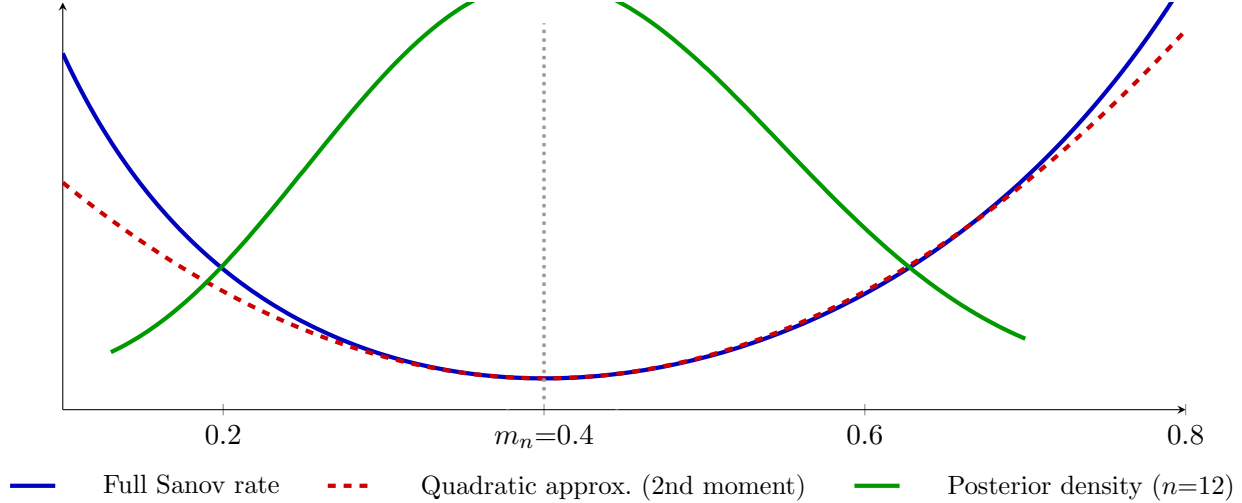

\section{Insufficiency of first-moment coherence}\label{sec:insuff}

\subsection{Non-identification and moment insufficiency}

\begin{lemma}[Non-identification]\label{lem:non-id}
For any $m\in(0,1)$ and any integer $k\ge 2$, there exist probability
measures $\nu_1,\nu_2$ on $[0,1]$ with
$\int\theta\dx\nu_1=\int\theta\dx\nu_2=m$
but
$\int(1-\theta)^k\dx\nu_1\ne\int(1-\theta)^k\dx\nu_2$.
\end{lemma}

\begin{proof}
Let $\nu_1=\delta_m$, so $\int(1-\theta)^k\dx\nu_1=(1-m)^k$.
Choose $0<a<m<b<1$ and set $\nu_2=\lambda\delta_a+(1-\lambda)\delta_b$
with $\lambda=(b-m)/(b-a)\in(0,1)$, ensuring $\int\theta\dx\nu_2=m$.
Since $\theta\mapsto(1-\theta)^k$ is strictly convex for $k\ge 2$,
Jensen's inequality gives
$\int(1-\theta)^k\dx\nu_2>(1-m)^k=\int(1-\theta)^k\dx\nu_1$.
\end{proof}

\begin{remark}[Extension to general functionals]\label{rem:general-nonlinear}
Lemma~\ref{lem:non-id} extends to any non-affine continuous
$f:[0,1]\to\mathbb{R}$: the mean does not determine
$\Ebb[f(\theta)\mid\mathcal{F}_n]$ (Appendix~\ref{app:general-nonlinear}).
\end{remark}

\subsection{Numerical illustration}

At $n=0$, take $\Pi_1=\BetaD(2,2)$ and
$\Pi_2=\BetaD(1,1)$.  Both have mean $1/2$.  Since
$\Ebb_{\BetaD(a,b)}[\theta^2]=a(a+1)/\bigl((a+b)(a+b+1)\bigr)$:
\begin{align*}
\Ebb_{\Pi_1}[(1-\theta)^2]&=\tfrac{3}{10},\qquad
\Ebb_{\Pi_2}[(1-\theta)^2]=\tfrac{1}{3}.
\end{align*}
Equal prior means; different two-step predictives.  The variances are
$1/20$ and $1/12$ respectively.  These priors do \emph{not} maintain
equal posterior means for $n\ge 1$: the posterior mean
$(a+S_n)/(a+b+n)$ depends on $a+b$.

\begin{theorem}[Moment insufficiency]\label{thm:insuff}
Let $k\ge 2$, $n\ge 0$, and $m_n\coloneqq
\Ebb[\theta\mid\mathcal{F}_n]\in(0,1)$.  The mapping
$m_n\mapsto\Ebb[(1-\theta)^k\mid\mathcal{F}_n]$ is set-valued:
distinct posteriors sharing mean $m_n$ yield distinct $k$-step predictives.
\end{theorem}

\begin{proof}
Fix $m_n\in(0,1)$ and $k\ge 2$.  We construct two posteriors sharing
mean $m_n$ but yielding distinct $k$-step predictives.

Let $\nu_1=\delta_{m_n}$, so
$\int(1-\theta)^k\dx\nu_1=(1-m_n)^k$.

Choose $0<a<m_n<b<1$ and set $\nu_2=\lambda\delta_a+(1-\lambda)\delta_b$
with $\lambda=(b-m_n)/(b-a)$.  Then $\int\theta\dx\nu_2=m_n$ by
construction, but $\nu_2$ has variance
$\sigma^2=\lambda(1-\lambda)(b-a)^2>0$.  Since
$\theta\mapsto(1-\theta)^k$ is strictly convex for $k\ge 2$,
Jensen's inequality gives
$\int(1-\theta)^k\dx\nu_2>(1-m_n)^k=\int(1-\theta)^k\dx\nu_1$.

By Theorem~\ref{thm:kreps}, the $k$-step run probability under
posterior $\nu_i$ is
$\Ebb[(1-\theta)^k\mid\mathcal{F}_n]=\int(1-\theta)^k\dx\nu_i$,
so the two posteriors yield distinct $k$-step predictives despite
sharing the same mean.  Since the choice of $a,b$ is arbitrary
(subject to $0<a<m_n<b<1$), the mapping
$m_n\mapsto\Ebb[(1-\theta)^k\mid\mathcal{F}_n]$ is set-valued.
\end{proof}

\subsection{Quantitative bound}

\begin{proposition}[Second-order discrepancy]\label{prop:quantitative}
Let $k\ge 2$ and let $\Pi(\cdot\mid\mathcal{F}_n)$ have mean $m_n$
and variance $\sigma_n^2$.  Then
\begin{equation}\label{eq:quantitative-bound}
\Ebb[(1-\theta)^k\mid\mathcal{F}_n]-(1-m_n)^k
=\frac{k(k-1)}{2}(1-\xi_n)^{k-2}\,\sigma_n^2
\end{equation}
for some $\xi_n\in[0,1]$.  In particular,
\begin{equation}\label{eq:quantitative-bounds}
0<\Ebb[(1-\theta)^k\mid\mathcal{F}_n]-(1-m_n)^k
\le\frac{k(k-1)}{2}\,\sigma_n^2.
\end{equation}
\end{proposition}

\begin{proof}
Set $f(\theta)=(1-\theta)^k$.  By Taylor's theorem with Lagrange
remainder, for each $\theta\in[0,1]$ there exists
$\xi(\theta)\in[\min(\theta,m_n),\,\max(\theta,m_n)]$ such that
$f(\theta)=f(m_n)+f'(m_n)(\theta-m_n)
+\tfrac12 f''(\xi(\theta))(\theta-m_n)^2$.
Take conditional expectations.  Since
$\Ebb[\theta-m_n\mid\mathcal{F}_n]=0$, the linear term vanishes:
\[
\Ebb[f(\theta)\mid\mathcal{F}_n]-f(m_n)
=\tfrac12\,\Ebb\bigl[f''(\xi(\theta))\,(\theta-m_n)^2\mid\mathcal{F}_n\bigr].
\]
Now $f''(\theta)=k(k-1)(1-\theta)^{k-2}$ is continuous and nonnegative
on $[0,1]$, with $0\le f''(\theta)\le k(k-1)$.  Since
$f''(\xi(\theta))(\theta-m_n)^2\ge 0$, the integral on the right is
nonnegative.  By the mean value theorem for integrals (applied to the
probability measure
$(\theta-m_n)^2\dx\Pi(\theta\mid\mathcal{F}_n)/\sigma_n^2$ when
$\sigma_n^2>0$), there exists $\xi_n\in[0,1]$ such that
\[
\Ebb\bigl[f''(\xi(\theta))\,(\theta-m_n)^2\mid\mathcal{F}_n\bigr]
=f''(\xi_n)\,\sigma_n^2
=k(k-1)(1-\xi_n)^{k-2}\,\sigma_n^2.
\]
This gives~\eqref{eq:quantitative-bound}.  The upper
bound~\eqref{eq:quantitative-bounds} follows from
$(1-\xi_n)^{k-2}\le 1$.  The lower bound in~\eqref{eq:quantitative-bounds}
is strict when $\sigma_n^2>0$ because $f''> 0$ on $(0,1)$.
\end{proof}

\begin{corollary}[Plug-in as Dirac]\label{cor:plugin-dirac}
The plug-in predictive
$\hat{p}_k^{\mathrm{plug}}\coloneqq(1-\theta_n)^k$ equals
$\Ebb[(1-\theta)^k\mid\mathcal{F}_n]$ under $\delta_{\theta_n}$.  For
$k\ge 2$, it systematically underestimates the run probability by at
most $\frac{k(k-1)}{2}\sigma_n^2$.  At $k=2$, the discrepancy equals
$\sigma_n^2$ exactly by~\eqref{eq:k2-identity}.
\end{corollary}

\begin{remark}[Strength of non-identification]\label{rem:strength}
The non-identification in Lemma~\ref{lem:non-id} is not a pathological
boundary phenomenon.  For any $m\in(0,1)$ and any $k\ge 2$, as the
posterior ranges over all distributions with mean~$m$, the $k$-step
predictive $\Ebb[(1-\theta)^k\mid\mathcal{F}_n]$ takes all values in
the interval $\bigl[(1-m)^k,\;1-m\bigr]$.  The lower bound
$(1-m)^k$ is attained by $\delta_m$ (the plug-in value, via Jensen's
inequality and strict convexity of $(1-\theta)^k$).  The upper bound
$1-m$ is attained by the two-point distribution
$(1-m)\delta_0+m\delta_1$, which has mean~$m$ and gives
$\Ebb[(1-\theta)^k]=(1-m)\cdot 1^k+m\cdot 0^k=1-m$.
At $k=2$ with $m=1/2$, the predictive ranges from $1/4$
(degenerate) to $1/2$ (endpoint distribution), a factor of two.
\end{remark}

\begin{figure}[htbp]
\centering
\includegraphics[width=\textwidth]{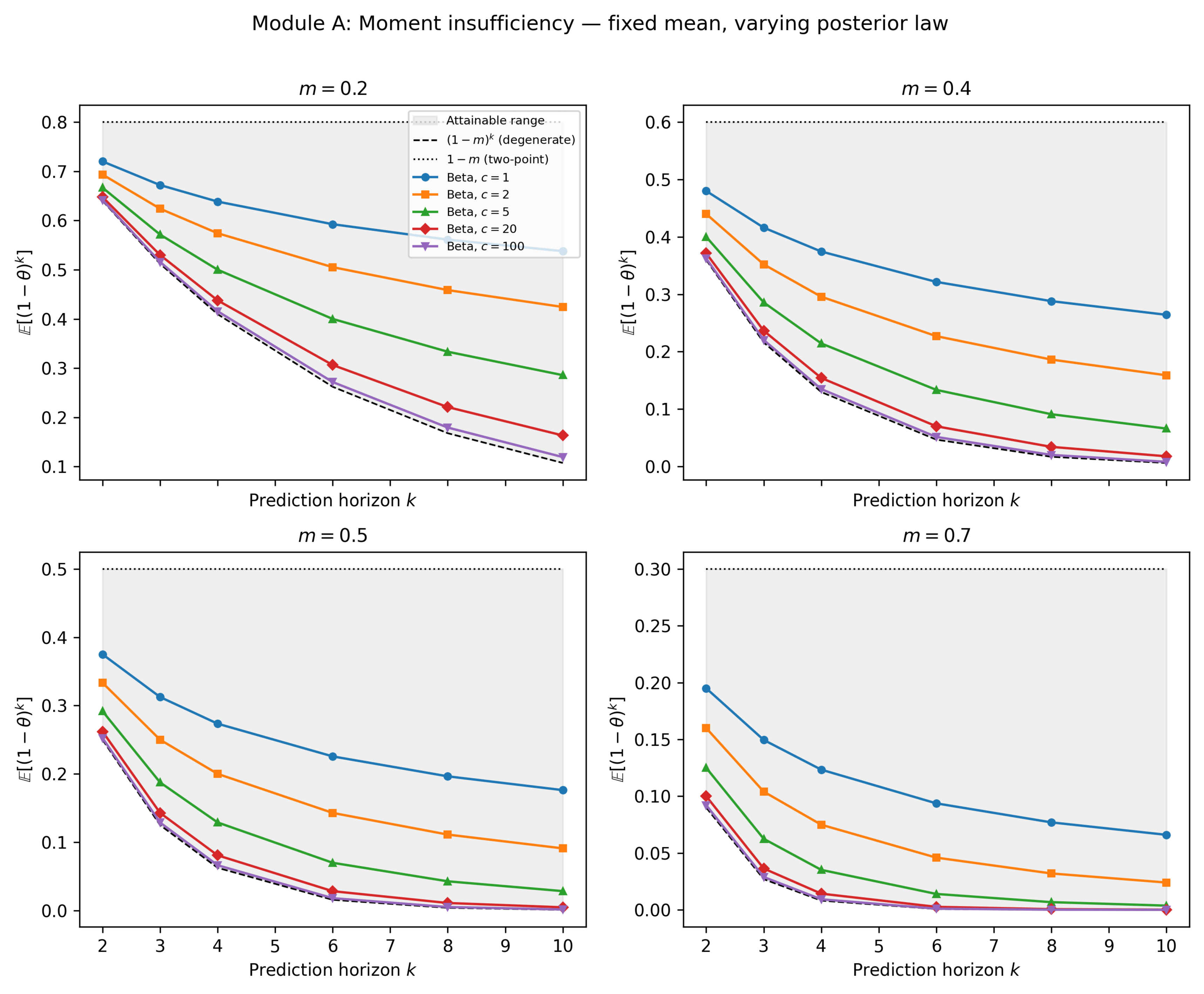}
\caption{Moment insufficiency: for fixed posterior mean $m$, the $k$-step
predictive $\Ebb[(1-\theta)^k]$ ranges over the shaded region
$\bigl[(1-m)^k,\;1-m\bigr]$
(Remark~\ref{rem:strength}).  Dashed: degenerate
lower bound $(1-m)^k$ (plug-in).  Dotted: two-point upper bound $1-m$.
Curves show Beta$(cm,\,c(1-m))$ posteriors at concentrations
$c\in\{1,2,5,20,100\}$.  Higher concentration drives the predictive
toward the degenerate plug-in; lower concentration spreads posterior
mass and increases the run probability.  All values are computed exactly
via $\Ebb[(1-\theta)^k]=(b)_k/(a+b)_k$.}
\label{fig:moment-insuff}
\end{figure}

\begin{remark}[Scope]\label{rem:scope}
Theorem~\ref{thm:insuff} does not assert that the martingale posterior
framework is incompatible with Bayesian inference.  Every Bayesian
posterior mean is a martingale (Remark~\ref{rem:bayes-special}).  The
theorem asserts that condition~\eqref{eq:mp-condition} alone does not
determine the full posterior.  A specific likelihood, a named prior, or
explicit higher-moment constraints are needed for multi-step predictive
coherence (Definition~\ref{def:pred-coherence}).
\end{remark}

\section{Decision-theoretic consequence}\label{sec:decision}

The insufficiency theorem has an immediate decision-theoretic corollary:
under any strictly proper scoring rule, the plug-in predictive is
inadmissible.  Let $\delta$ be a predictive rule producing
$\hat{P}_\delta(\cdot\mid\mathcal{F}_n)$ on $\{0,1\}^k$.

\begin{definition}\label{def:scoring}
A \emph{scoring rule} $S(\hat{P},x)$ assigns a loss to predictive
$\hat{P}$ when outcome $x$ is realized.  It is \emph{strictly proper} if
$\Ebb_P[S(\hat{P},X)]$ is uniquely minimized at $\hat{P}=P$.
The \emph{risk} is
$R(\delta,\theta)=\Ebb_\theta[S(\hat{P}_\delta,X_{n+1:n+k})]$.
A rule is \emph{admissible} if no rule uniformly dominates it in risk.
\end{definition}

The log score $S(\hat{P},x)=-\log\hat{P}(x)$ and Brier score
$S(\hat{P},x)=\sum_y(\ind[y=x]-\hat{P}(y))^2$ are both strictly
proper~\citep{gneiting2007strictly}.

\subsection{Complete class theorem}

The following result is classical and included for completeness;
see Blackwell and Girshick~(\citeyear{blackwell1954theory}) and
Wald~(\citeyear{wald1950statistical}).

\begin{theorem}[Blackwell--Girshick]\label{thm:complete-class}
Let $S$ be a strictly proper scoring rule with parameter space $[0,1]$
and risk continuous in~$\theta$.  Then:
\begin{enumerate}[label=(\alph*),nosep]
\item Every Bayes rule is admissible.
\item Every admissible rule is Bayes or a limit of Bayes rules.
\end{enumerate}
The unique Bayes-optimal predictive under prior $\Pi$ is the posterior
predictive $\hat{P}^{\mathrm{Bayes}}(x\mid\mathcal{F}_n)=
\Ebb[\theta^s(1-\theta)^{k-s}\mid\mathcal{F}_n]$.
\end{theorem}

\begin{proof}
Part~(a) follows by integrating against $\Pi$: if $\delta'$ dominated
$\delta^*$, it would have lower Bayes risk, contradicting optimality.
Part~(b) follows from compactness of $[0,1]$ and continuity of risk;
see Blackwell and Girshick~(\citeyear{blackwell1954theory}).
The Bayes-optimal predictive is the unique minimiser of expected
score under strict propriety, which by Theorem~\ref{thm:kreps}
equals $\Ebb[\theta^s(1-\theta)^{k-s}\mid\mathcal{F}_n]$.
\end{proof}

\subsection{Plug-in dominance}

\begin{proposition}[Plug-in is strictly dominated]\label{prop:plugin-dominated}
Let $\Pi$ have full support on $[0,1]$ and let $S$ be any strictly
proper scoring rule.  For $k\ge 2$, the plug-in rule
$\delta^{\mathrm{plug}}$ is strictly dominated by
$\delta^{\mathrm{Bayes}}$: for every $\theta\in(0,1)$,
$R(\delta^{\mathrm{Bayes}},\theta)<R(\delta^{\mathrm{plug}},\theta)$.
\end{proposition}

\begin{proof}
We show that plug-in and Bayes predictives differ almost
surely, then apply strict propriety.

Fix $X_{1:n}$ with $\Var(\theta\mid\mathcal{F}_n)>0$ (which holds
$\Pbb_\theta$-a.s.\ since $\Pi$ has full support).  The Bayes predictive
is $\hat{P}^{\mathrm{Bayes}}(x)=\Ebb[\theta^s(1-\theta)^{k-s}\mid\mathcal{F}_n]$;
the plug-in is $\hat{P}^{\mathrm{plug}}(x)=\theta_n^s(1-\theta_n)^{k-s}$.
By Corollary~\ref{cor:plugin-dirac}, for $k\ge 2$ these differ on a
set of positive probability whenever the posterior is non-degenerate.
Since the Bayes predictive equals the conditional distribution of the
future block given $\mathcal{F}_n$, strict propriety implies that
$\hat{P}^{\mathrm{Bayes}}$ strictly minimises expected score.  Hence the
plug-in's expected score is strictly higher.  Integrating over
$\mathcal{F}_n$ under $\Pbb_\theta$ preserves the strict inequality.
\end{proof}

The dominance is strict whenever $\Var(\theta\mid\mathcal{F}_n)>0$.

\begin{corollary}[Inadmissibility]\label{cor:inadmissible}
The plug-in rule is inadmissible for $k$-step prediction with $k\ge 2$
under any strictly proper scoring rule.
\end{corollary}

\begin{proof}
By Proposition~\ref{prop:plugin-dominated}, for every
$\theta\in(0,1)$, the Bayes predictive achieves strictly lower risk
than the plug-in predictive under any strictly proper scoring rule.
Since the Bayes rule uniformly dominates the plug-in and is itself
admissible (Theorem~\ref{thm:complete-class}(a)), the plug-in is
inadmissible.
\end{proof}

\begin{corollary}[Mean-only rules]\label{cor:mean-only-class}
Let $k\ge 2$.  Any predictive rule depending on $\mathcal{F}_n$ only
through $\theta_n$ is either a Bayes rule under a degenerate prior, or
outside the Bayes class and hence dominable.
\end{corollary}

\begin{proof}
By Theorem~\ref{thm:insuff}, there exist posteriors $\nu_1,\nu_2$ with
$\int\theta\dx\nu_1=\int\theta\dx\nu_2=m_n$ but
$\int(1-\theta)^k\dx\nu_1\ne\int(1-\theta)^k\dx\nu_2$ for $k\ge 2$.
A mean-only rule $\delta$ maps both posteriors to the same predictive
$\hat{P}_\delta(m_n)$.  By strict propriety, the unique Bayes-optimal
predictive under $\nu_i$ is
$\hat{P}^*_i(x)=\int\theta^s(1-\theta)^{k-s}\dx\nu_i$, and
$\hat{P}^*_1\ne\hat{P}^*_2$.  Since $\delta$ assigns the same
$\hat{P}$ to both, it cannot equal both $\hat{P}^*_1$ and $\hat{P}^*_2$;
hence $\delta$ is suboptimal under at least one posterior and is
outside the Bayes class (unless $\nu_1$ or $\nu_2$ is degenerate).
\end{proof}

\begin{remark}[Log-score risk gap]\label{rem:risk-gap-log}
At $k=2$, the Bayes predictive assigns run probability
$p_B=\Ebb[(1-\theta)^2\mid\mathcal{F}_n]=(1-m_n)^2+\sigma_n^2$ and
the plug-in assigns $p_P=(1-m_n)^2$.  Since $p_B\ne p_P$ whenever
$\sigma_n^2>0$, the KL divergence
$D_{\mathrm{KL}}(\text{Bayes}\|\text{plug-in})$ is strictly positive,
so the log-score risk gap is strictly positive whenever the posterior is
non-degenerate.
\end{remark}

\begin{example}[Numerical risk gap]\label{ex:risk-gap-numerical}
Continue the setting of Table~\ref{tab:comparison}: $n=5$, $S_5=2$,
Jeffreys prior, so $m_5=5/12$ and $\sigma_5^2\approx 0.035$.

At $k=2$, the Bayes predictive for the all-zeros run is
$p_B=(b_5)_2/(N_5)_2=0.375$ and the plug-in is
$p_P=(1-m_5)^2\approx 0.340$.  The log-score risk gap for predicting
the binary event $Y=\ind[X_6=X_7=0]$ is
\begin{align*}
\Delta R_{\log}
&=\Ebb_Y[-\log p_P(Y)]-\Ebb_Y[-\log p_B(Y)]\\
&=p_B\log\frac{p_B}{p_P}+(1-p_B)\log\frac{1-p_B}{1-p_P}\\
&\approx 0.375\log\frac{0.375}{0.340}+0.625\log\frac{0.625}{0.660}
\approx 0.0026.
\end{align*}
This KL divergence, while small in absolute terms, is strictly positive
and cannot be eliminated by any mean-only adjustment.  At $k=4$ the
divergence grows to approximately $0.021$, an order of magnitude
larger.
\end{example}

\begin{remark}[Brier score risk gap]\label{rem:risk-gap-brier}
Under the Brier score $S(\hat{p},y)=(\hat{p}-y)^2$ for predicting the
binary event $Y=\ind[X_{n+1}=\cdots=X_{n+k}=0]$, the risk gap between
plug-in and Bayes is $(p_B-p_P)^2$ to leading order.  At $k=2$, this
equals $\sigma_n^4$.  The qualitative conclusion (plug-in is strictly
dominated whenever $\sigma_n^2>0$) holds under all strictly proper
scoring rules.
\end{remark}

Under the log score, the plug-in--Bayes gap equals
$D_{\mathrm{KL}}(\text{Bayes}\|\text{plug-in})$~\citep{cover2006elements},
connecting predictive regret to KL divergence.

\begin{figure}[htbp]
\centering
\includegraphics[width=\textwidth]{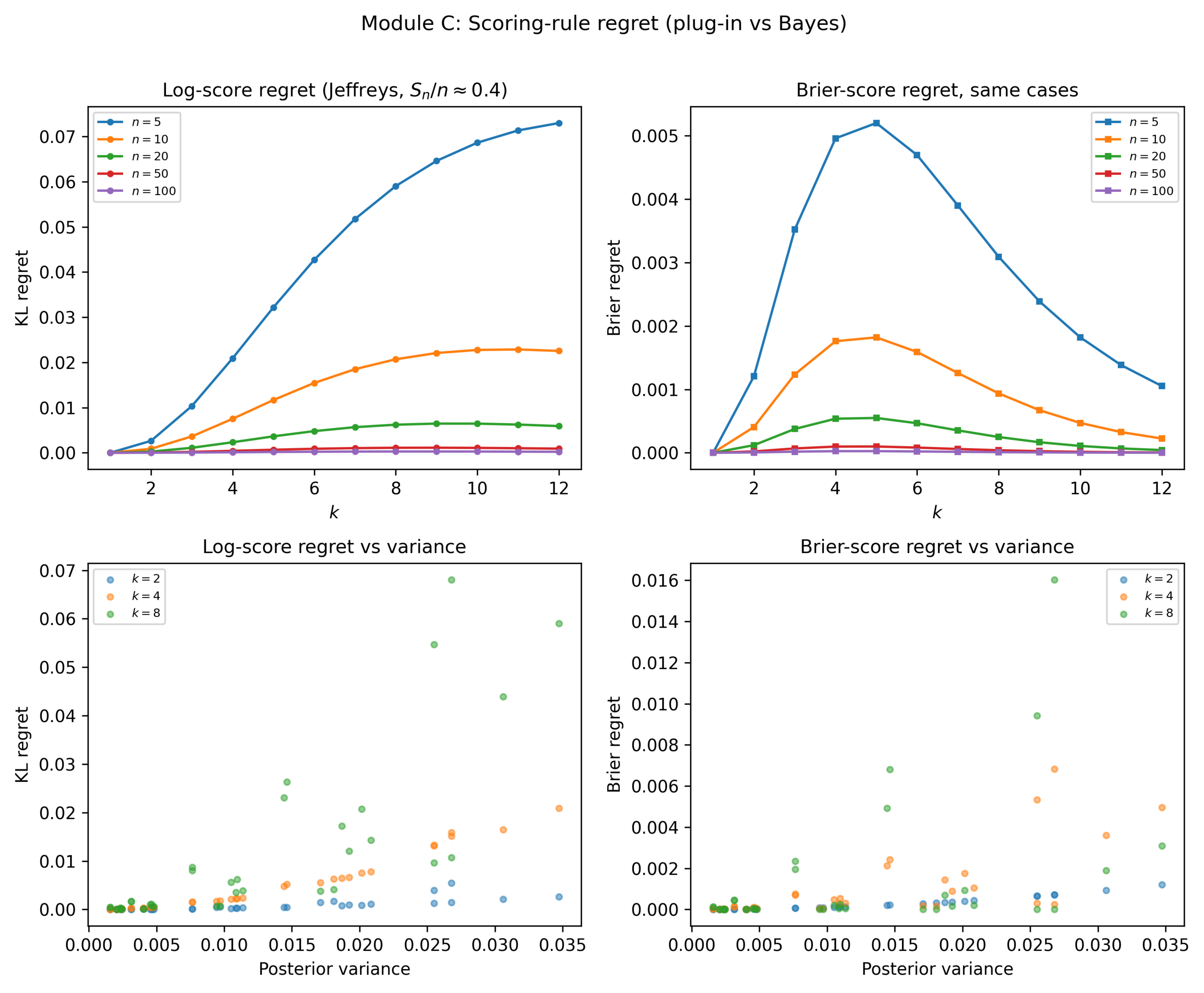}
\caption{Scoring-rule regret (plug-in minus Bayes).
\textit{Top left}: KL regret
$D_{\mathrm{KL}}(\hat{P}^{\mathrm{Bayes}}\|\hat{P}^{\mathrm{plug}})$
vs.\ horizon $k$ at $S_n/n\approx 0.4$ under the Jeffreys prior, for
sample sizes $n\in\{5,10,20,50,100\}$.  Regret grows with $k$ (more
moments neglected) and shrinks with $n$ (posterior concentrates).
\textit{Top right}: Brier regret $(p_B-p_P)^2$, same cases.
\textit{Bottom row}: regret vs.\ posterior variance across all
prior/sample-size combinations, showing the $O(\sigma_n^4)$ Brier
scaling (Remark~\ref{rem:risk-gap-brier}) and the strictly positive
KL gap whenever $\sigma_n^2>0$
(Remark~\ref{rem:risk-gap-log}).  All values are exact.}
\label{fig:scoring-regret}
\end{figure}

\section{Hill's \texorpdfstring{$A_{(n)}$}{A(n)} as a positive example}\label{sec:hill}

Hill's $A_{(n)}$~\citep{hill1968posterior,hill1988deFinetti}, in the
Bernoulli case, is the posterior predictive under the Jeffreys prior
$\BetaD(\tfrac12,\tfrac12)$.

\begin{proposition}[Hill/Jeffreys one-step predictive]\label{prop:hill-onestep}
Let $\theta\sim\BetaD(\tfrac12,\tfrac12)$ and
$X_i\mid\theta\stackrel{\mathrm{i.i.d.}}{\sim}\Bern(\theta)$.
The posterior is $\BetaD(\tfrac12+S_n,\,\tfrac12+n-S_n)$, and
$\Pbb(X_{n+1}=1\mid\mathcal{F}_n)=(S_n+\tfrac12)/(n+1)$.
\end{proposition}

\begin{proof}
The posterior mean of $\BetaD(a,b)$ is $a/(a+b)$.  With
$a=\tfrac12+S_n$ and $b=\tfrac12+n-S_n$, the result follows by
Theorem~\ref{thm:kreps} with $k=1$.
\end{proof}

\begin{proposition}[$k$-step run probabilities]\label{prop:hill-kstep}
Under the same model, with $a_n=S_n+\tfrac12$, $b_n=n-S_n+\tfrac12$,
$N_n=a_n+b_n=n+1$, and $(x)_k=x(x+1)\cdots(x+k-1)$ the rising factorial,
\[
\Pbb(X_{n+1}=\cdots=X_{n+k}=0\mid\mathcal{F}_n)
=\prod_{j=0}^{k-1}\frac{b_n+j}{N_n+j}.
\]
\end{proposition}

\begin{proof}
By Theorem~\ref{thm:kreps} and the Beta moment formula:
$\Ebb[(1-\theta)^k\mid\mathcal{F}_n]
=B(a_n,b_n+k)/B(a_n,b_n)=(b_n)_k/(N_n)_k$.
\end{proof}

\subsection{Full moment recovery}

Under $\BetaD(a_n,b_n)$, all posterior moments are determined:
\begin{equation}\label{eq:beta-moments}
\Ebb[\theta^j\mid\mathcal{F}_n]
=\frac{(a_n)_j}{(N_n)_j}
=\prod_{i=0}^{j-1}\frac{a_n+i}{N_n+i},
\qquad j\ge 1.
\end{equation}
Hill's rule specifies every conditional moment as an explicit function
of $(n,S_n)$.  It is a Bayes rule and therefore admissible
(Theorem~\ref{thm:complete-class}(a)).

The first three moments with $N_n=n+1$ are:
\begin{align}
\mu_1&=\frac{a_n}{N_n}=\frac{S_n+\tfrac12}{n+1},\label{eq:jeff-mu1}\\
\mu_2&=\frac{a_n(a_n+1)}{N_n(N_n+1)}
=\frac{(S_n+\tfrac12)(S_n+\tfrac32)}{(n+1)(n+2)},\label{eq:jeff-mu2}\\
\mu_3&=\frac{a_n(a_n+1)(a_n+2)}{N_n(N_n+1)(N_n+2)}
=\frac{(S_n+\tfrac12)(S_n+\tfrac32)(S_n+\tfrac52)}{(n+1)(n+2)(n+3)}.
\label{eq:jeff-mu3}
\end{align}
Each is a rational function of $(n,S_n)$; the posterior variance is
$\sigma_n^2=\mu_2-\mu_1^2=a_n b_n/(N_n^2(N_n+1))$.

\subsection{Plug-in versus Bayes comparison}

Table~\ref{tab:comparison} compares the plug-in and Bayes (Jeffreys)
$k$-step run probabilities $\Pbb(X_{n+1:n+k}=0\mid\mathcal{F}_n)$
for $n=5$, $S_5=2$ ($m_5=5/12\approx 0.417$).

\begin{table}[htbp]
\centering
\caption{Plug-in vs.\ Bayes (Jeffreys) $k$-step run probabilities
$\Pbb(X_{n+1:n+k}=0\mid\mathcal{F}_n)$, $n=5$, $S_5=2$.}\label{tab:comparison}
\begin{tabular}{@{}cccc@{}}
\toprule
$k$ & Plug-in $(1-m_5)^k$ & Bayes $(b_5)_k/(N_5)_k$
& Relative gap \\
\midrule
2 & 0.340 & 0.375 & 9.3\% \\
3 & 0.199 & 0.258 & 23.0\% \\
4 & 0.116 & 0.186 & 37.8\% \\
\bottomrule
\end{tabular}
\end{table}

The Jeffreys posterior variance is
$\sigma_5^2=a_5 b_5/(N_5^2(N_5+1))=(2.5)(3.5)/(252)\approx 0.035$,
matching the $k=2$ discrepancy by~\eqref{eq:k2-identity}.

To verify: with $a_5=2.5$, $b_5=3.5$, $N_5=6$, the Bayes predictives are
$(b_5)_k/(N_5)_k$:
\begin{align*}
k=2:&\quad(3.5)(4.5)/\bigl((6)(7)\bigr)=0.375,\\
k=3:&\quad(3.5)(4.5)(5.5)/\bigl((6)(7)(8)\bigr)\approx 0.258,\\
k=4:&\quad(3.5)(4.5)(5.5)(6.5)/\bigl((6)(7)(8)(9)\bigr)\approx 0.186.
\end{align*}
The plug-in values are $(7/12)^k$: $0.340$, $0.199$, $0.116$.  The
relative gap grows from $9.3\%$ ($k=2$) to $37.8\%$ ($k=4$), reflecting
accumulating dependence on higher-order moments: for $k\ge 3$, the
second-order Taylor bound $\frac{k(k-1)}{2}\sigma_n^2$ underestimates the
true discrepancy.

\begin{figure}[htbp]
\centering
\includegraphics[width=\textwidth]{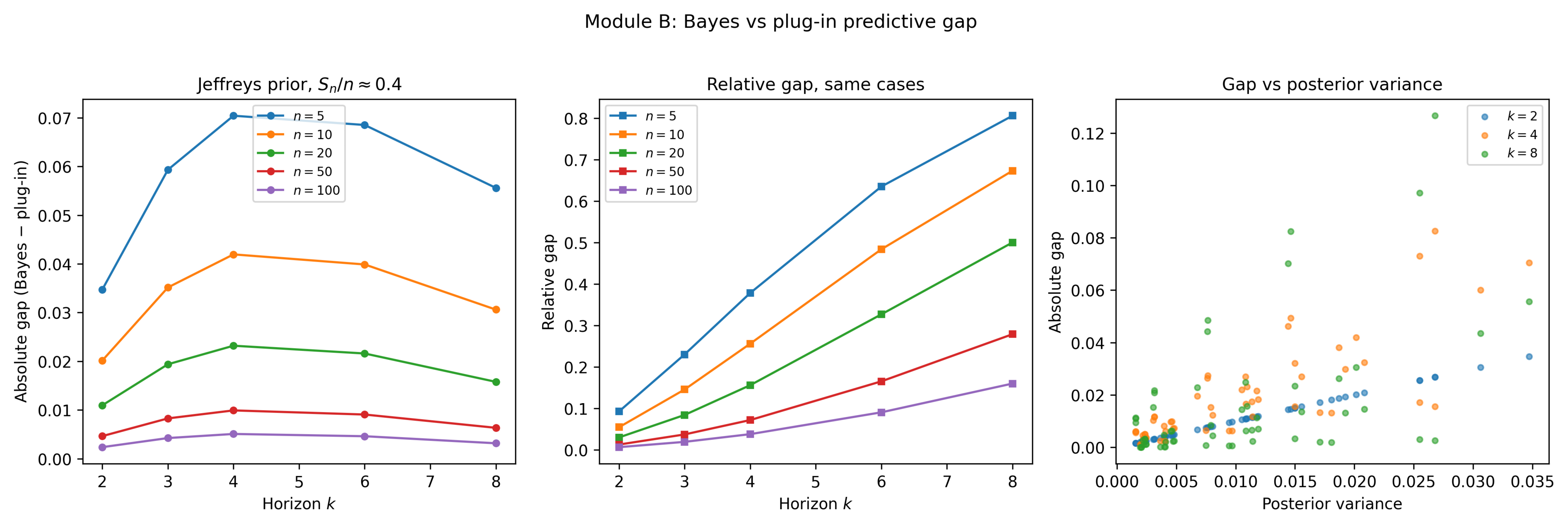}
\caption{Bayes vs.\ plug-in predictive gap under Beta-Bernoulli
updating.  \textit{Left}: absolute gap
$\Ebb[(1-\theta)^k\mid\mathcal{F}_n]-(1-m_n)^k$ vs.\ horizon $k$ for the
Jeffreys prior at $S_n/n\approx 0.4$, extending Table~\ref{tab:comparison}
across sample sizes $n\in\{5,10,20,50,100\}$.  The gap grows with $k$
(each step adds a new undetermined moment) and shrinks with $n$ (posterior
concentrates).  \textit{Centre}: the corresponding relative gap, which
exceeds $50\%$ at $k=8$ for $n=5$, illustrating that the plug-in can
substantially underestimate the Bayes predictive at moderate horizons.
\textit{Right}: absolute gap vs.\ posterior variance across all priors and
configurations, confirming the $O(\sigma_n^2)$ scaling from
Proposition~\ref{prop:quantitative}.  All values are computed exactly via
$(b_n)_k/(N_n)_k$.}
\label{fig:bayes-plugin-gap}
\end{figure}

\begin{remark}[Lane--Sudderth caveat]\label{rem:lane-sudderth}
Hill's $A_{(n)}$ was formulated nonparametrically.  Lane and
Sudderth~(\citeyear{lane1984conditionally}) showed that in nonparametric
settings, conglomerability can fail.  In the Bernoulli case,
$\BetaD(\tfrac12,\tfrac12)$ is a proper prior and no such issue arises.
\end{remark}

\begin{remark}[Red/black card prediction]\label{rem:red-black}
In the standard deck-prediction problem, one-step prediction depends
only on the posterior mean (martingale coherence), but $k$-step run
probabilities require $\Ebb[(1-\theta)^k\mid\mathcal{F}_n]$, hence all
posterior moments up to order~$k$ (Corollary~\ref{cor:moment-hierarchy}).
The curvature of $(1-\theta)^k$ makes the variance contribution
non-negligible whenever the remaining composition deviates from~$1/2$.
\end{remark}

\section{Asymptotic reconciliation}\label{sec:asymp}

The moment insufficiency of Sections~\ref{sec:insuff}--\ref{sec:decision} is a
\emph{finite-sample structural phenomenon}.  It does not preclude asymptotic
proximity between empirical and predictive measures, which has been studied
extensively.

Berti, Crimaldi, Pratelli, and Rigo~(\citeyear{berti2009rate}) study
the empirical process
$C_n(B)=\sqrt{n}\{\mu_n(B)-\Pbb(X_{n+1}\in B\mid\mathcal{F}_n)\}$
for c.i.d.\ and exchangeable sequences, obtaining stable convergence
at rate $n^{-1/2}$.
Berti, Pratelli, and Rigo~(\citeyear{berti2004limit}) establish
law-level rates using bounded Lipschitz metrics, and
\citet{berti2017rate} give matching bounds for
finite alphabets: in the Bernoulli case,
\begin{equation}\label{eq:bpr-bounds}
\frac{a}{n}\;\le\;\rho\bigl[\mathcal{L}(\mu_n),\,\mathcal{L}(\mu)\bigr]
\;\le\;\frac{b}{n+1},
\end{equation}
where $a=\Ebb[\theta(1-\theta)]$ and $b$ depends on
the total variation of the prior density.  The lower bound is
driven by the expected posterior variance, the same
quantity that governs the moment-hierarchy discrepancy in
Proposition~\ref{prop:quantitative}.
See also \citet{berti2018asymptotic} and
\citet{leisen2025weak} for further asymptotic developments.

Our concern is complementary: we ask what finite-horizon predictive
content is \emph{identified by first-moment information itself},
independent of any asymptotic limit.  The bounds
of~(\citeyear{berti2009rate}) and~(\citeyear{berti2017rate}) measure
how fast empirical and predictive \emph{measures} approach one another;
our results identify which \emph{functionals} of the posterior (moments
of order $\ge 2$) remain structurally underdetermined by the first
moment alone.  Despite this structural incompleteness, posterior
concentration brings plug-in and Bayes predictives close together
asymptotically, though this closeness does not erase the hierarchy.

\begin{proposition}[Asymptotic agreement]\label{prop:asymp-agree}
Let $k\ge 1$ be fixed and let
$\sigma_n^2=\Var(\theta\mid\mathcal{F}_n)\to 0$ a.s.  Then
\[
\bigl|\Ebb[(1-\theta)^k\mid\mathcal{F}_n]
-(1-\Ebb[\theta\mid\mathcal{F}_n])^k\bigr|
\le\tfrac{k(k-1)}{2}\,\sigma_n^2\to 0\quad\text{a.s.}
\]
\end{proposition}

\begin{proof}
By Proposition~\ref{prop:quantitative} with $f(\theta)=(1-\theta)^k$,
the discrepancy satisfies
$|\Ebb[(1-\theta)^k\mid\mathcal{F}_n]-(1-m_n)^k|
=\frac{k(k-1)}{2}(1-\xi_n)^{k-2}\sigma_n^2$
for some $\xi_n\in[0,1]$.  Since $(1-\xi_n)^{k-2}\le 1$, the bound
$\frac{k(k-1)}{2}\sigma_n^2$ holds.  As $\sigma_n^2\to 0$ a.s.\ by
assumption, the right-hand side vanishes for each fixed $k$, giving the
claimed convergence.
\end{proof}

Under $\BetaD(a,b)$ prior, the posterior variance is
\begin{equation}\label{eq:beta-variance}
\sigma_n^2
=\frac{(a+S_n)(b+n-S_n)}{(a+b+n)^2(a+b+n+1)}.
\end{equation}
Since $S_n/n\to\theta_0$ a.s., $\sigma_n^2\sim\theta_0(1-\theta_0)/n$,
confirming $\sigma_n^2=O(1/n)$.  The following proposition gives a
uniform-in-$k$ scaling result.

\begin{proposition}[Uniform bound]\label{prop:uniform-k}
Let $K_n=o(\sqrt{n})$.  Then
$\sup_{1\le k\le K_n}\bigl|\Ebb[(1-\theta)^k\mid\mathcal{F}_n]
-(1-m_n)^k\bigr|\to 0$ a.s.
\end{proposition}

\begin{proof}
By Proposition~\ref{prop:quantitative}, for each $k\le K_n$,
$|\Ebb[(1-\theta)^k\mid\mathcal{F}_n]-(1-m_n)^k|
\le\frac{k(k-1)}{2}\sigma_n^2$.
Taking the supremum over $1\le k\le K_n$ replaces $k(k-1)/2$ by its
maximum $K_n(K_n-1)/2\le K_n^2/2$, giving
$\sup_{k\le K_n}|\Ebb[(1-\theta)^k\mid\mathcal{F}_n]-(1-m_n)^k|
\le\frac{K_n^2}{2}\sigma_n^2$.
Under $\BetaD(a,b)$ (or any prior with $\sigma_n^2=O(1/n)$), the
right-hand side is $O(K_n^2/n)$, which converges to zero if and only
if $K_n=o(\sqrt{n})$.  This threshold reflects the tension between
growing curvature in $(1-\theta)^k$ and decreasing posterior variance.
\end{proof}

\begin{proposition}[Bernstein--von Mises rate]\label{prop:bvm-rate}
Let $\theta_0\in(0,1)$ and suppose that the prior~$\Pi$ admits a
continuous, strictly positive density in a neighbourhood of~$\theta_0$.
Under~$\Pbb_{\theta_0}$ (that is, when data are generated
i.i.d.\ $\Bern(\theta_0)$ and posterior inference uses the
prior~$\Pi$), for fixed~$k\ge 2$ and as $n\to\infty$,
\[
\Ebb[(1-\theta)^k\mid\mathcal{F}_n]-(1-m_n)^k
=O(n^{-1})\quad\Pbb_{\theta_0}\text{-a.s.}
\]
The discrepancy between the Bayes predictive and the plug-in predictive
thus vanishes at rate $O(1/n)$.
\end{proposition}

\begin{proof}
Under $\Pbb_{\theta_0}$, the observations are i.i.d.\
$\Bern(\theta_0)$, and the classical Bernstein--von~Mises theorem
\citep[Chapter~10]{vaart1998asymptotic} gives the posterior variance
$\sigma_n^2=O(n^{-1})$ $\Pbb_{\theta_0}$-a.s.; the regularity
conditions (continuous, strictly positive density near $\theta_0$)
are assumed.  By Proposition~\ref{prop:quantitative},
\[
0\;\le\;\Ebb[(1-\theta)^k\mid\mathcal{F}_n]-(1-m_n)^k
\;\le\;\frac{k(k-1)}{2}\,\sigma_n^2
\;=\;O(n^{-1}).
\]
The lower bound is strict when $\sigma_n^2>0$
(Proposition~\ref{prop:quantitative}), and the upper bound uses
$(1-\xi_n)^{k-2}\le 1$.
\end{proof}

\begin{remark}[Role of $\theta_0$]\label{rem:role-theta0}
The $O(n^{-1})$ rate in Proposition~\ref{prop:bvm-rate} holds under
$\Pbb_{\theta_0}$ for each fixed $\theta_0\in(0,1)$.  The implicit
constant depends on $\theta_0$ (through $\sigma_n^2\sim\theta_0(1-\theta_0)/n$)
and on $k$.  If $\Pi$ has full support, the regularity conditions hold
for every $\theta_0\in(0,1)$ simultaneously; there is no contradiction,
because for each $\theta_0$ the convergence is under the corresponding
product measure $\Pbb_{\theta_0}$, which generates different data
sequences and different posterior sequences.  The left-hand side is a
computable posterior functional that does not require knowledge
of~$\theta_0$, but its $\Pbb_{\theta_0}$-a.s.\ rate depends on the true
parameter through the data it generates.
\end{remark}

\begin{remark}[Finite-sample relevance]\label{rem:finite-sample}
The discrepancy vanishes as $n\to\infty$ for fixed~$k$.  When the
prediction horizon grows with $n$, Proposition~\ref{prop:uniform-k}
requires $k=o(\sqrt{n})$; for $k\ge C\sqrt{n}$, the bound
$\frac{k(k-1)}{2}\sigma_n^2$ remains $O(1)$.
In short, the plug-in and Bayes predictives agree uniformly over all
horizons that grow slower than~$\sqrt{n}$, but diverge once the
horizon reaches the~$\sqrt{n}$ scale: first-moment information
suffices for short-range prediction but not for long-range prediction,
even asymptotically.
\end{remark}

\begin{figure}[htbp]
\centering
\includegraphics[width=\textwidth]{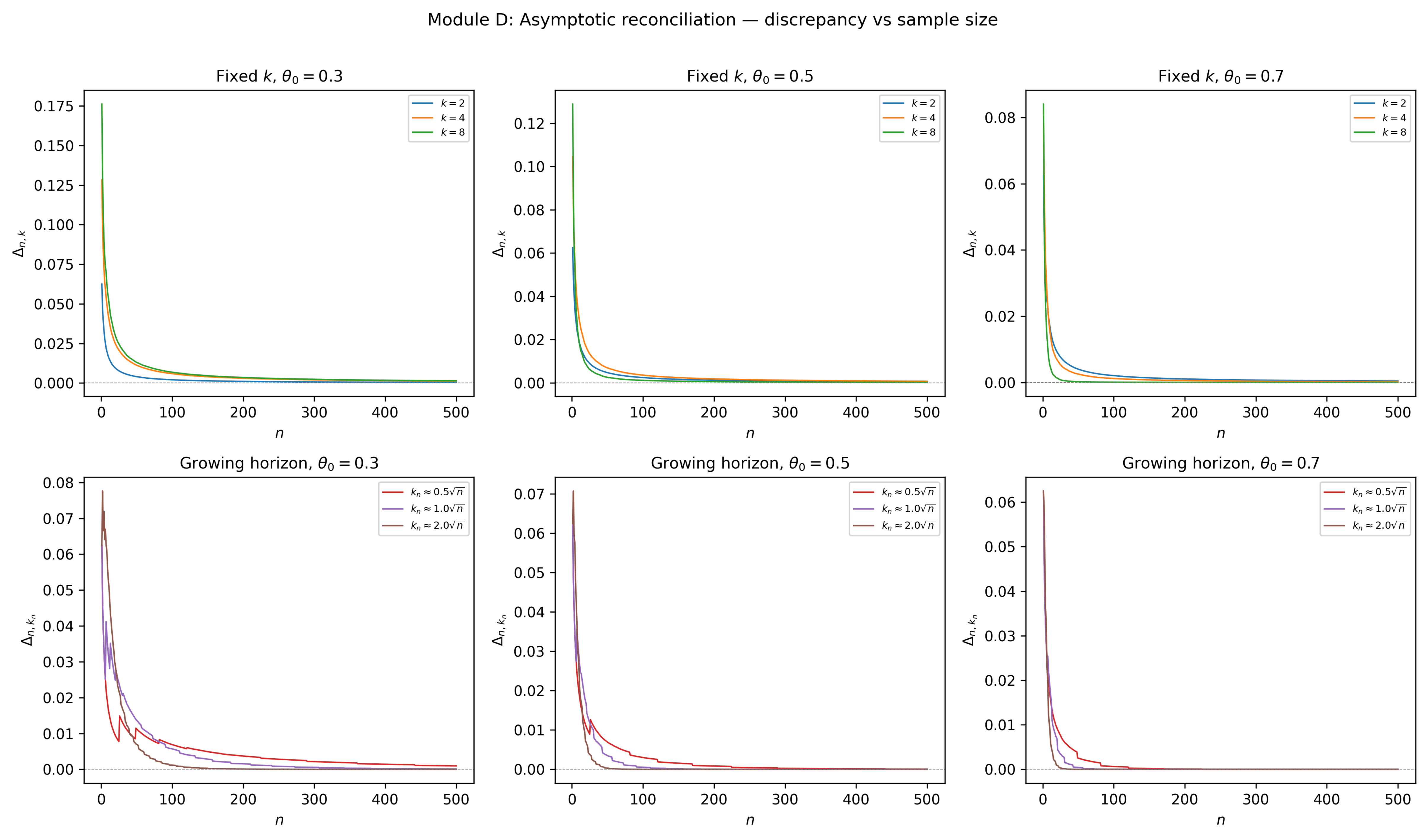}
\caption{Asymptotic reconciliation under Jeffreys prior, averaged over
200 simulated Bernoulli sequences.
\textit{Top row}: discrepancy $\Delta_{n,k}$ vs.\ sample size $n$ for
fixed horizons $k\in\{2,4,8\}$ at three true parameter values
$\theta_0\in\{0.3,0.5,0.7\}$.  The discrepancy vanishes at rate
$O(1/n)$ (Proposition~\ref{prop:bvm-rate}), with larger $k$ producing
a proportionally larger gap at each $n$.
\textit{Bottom row}: growing horizon $k_n\approx c\sqrt{n}$ for
$c\in\{0.5,1.0,2.0\}$.  When $k_n=o(\sqrt{n})$ the discrepancy
still vanishes; at $k_n\sim c\sqrt{n}$ it stabilises at a positive level,
illustrating the $\sqrt{n}$ threshold of
Proposition~\ref{prop:uniform-k}: first-moment information suffices for
short-range prediction but not for long-range prediction.}
\label{fig:asymptotic}
\end{figure}

\section{Structural hierarchy}\label{sec:hierarchy}

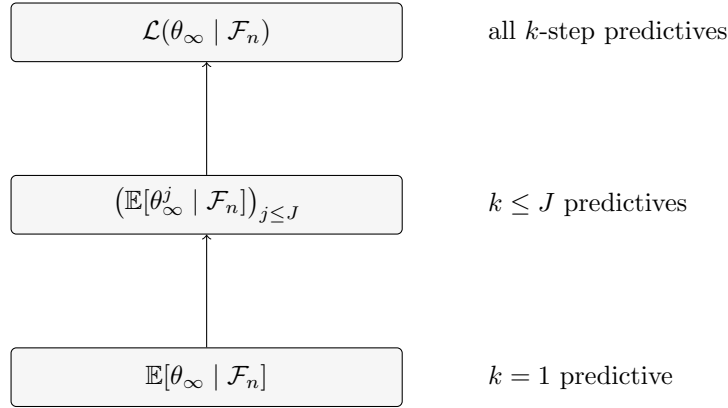
\begin{figure}[htbp]
\centering
\begin{tikzpicture}[
  node distance=1.5cm,
  box/.style={draw, rounded corners=2pt, fill=gray!7,
    minimum width=5.2cm, minimum height=0.78cm,
    align=center, font=\small},
  ann/.style={anchor=west, align=left, font=\small}]
\node[box] (L1) {$\Ebb[\theta_\infty\mid\mathcal{F}_n]$};
\node[box, above=of L1] (L2)
    {$\bigl(\Ebb[\theta_\infty^j\mid\mathcal{F}_n]\bigr)_{j\le J}$};
\node[box, above=of L2] (L3)
    {$\mathcal{L}(\theta_\infty\mid\mathcal{F}_n)$};
\draw[->] (L1) -- (L2);
\draw[->] (L2) -- (L3);
\node[ann, right=1.0cm of L1] {$k=1$ predictive};
\node[ann, right=1.0cm of L2] {$k\le J$ predictives};
\node[ann, right=1.0cm of L3] {all $k$-step predictives};
\end{tikzpicture}
\caption{The predictive moment hierarchy.  Bottom: the martingale /
mean-only tier (cf.\ linear Bayes) constrains only the first moment and determines
one-step predictives.  Middle: specifying $J$ moments
(Goldstein~(\citeyear{goldstein1975}) prevision tier) determines $k$-step
predictives for $k\le J$ but not for $k>J$ (Theorem~\ref{thm:insuff}).
Top: the full conditional law determines all $k$-step predictives
(predictive completeness; Theorem~\ref{thm:closure}).}
\label{fig:pred-hierarchy}
\end{figure}

\subsection{Hierarchy of predictive frameworks}

Table~\ref{tab:hierarchy} summarizes the relationship between the
structure imposed by a predictive framework and the moments it
determines.

\begin{table}[htbp]
\centering
\caption{Structural hierarchy for exchangeable Bernoulli
sequences.}\label{tab:hierarchy}
\begin{tabular}{@{}lll@{}}
\toprule
Framework & Structure imposed & Moments determined \\
\midrule
Martingale posterior & $\Ebb[\theta_\infty\mid\mathcal{F}_n]=\theta_n$
& 1st moment \\
$j$-moment specification & $\mu_1,\ldots,\mu_j$ constrained
& up to $j$-th moment \\
Full Bayesian posterior & $\Pi(\cdot\mid\mathcal{F}_n)$ specified
& all moments \\
Conformal / rank calibration & marginal coverage & none \\
\bottomrule
\end{tabular}
\end{table}

The hierarchy is strict: specifying $j$ moments determines $k$-step
predictives for $k\le j$ (Corollary~\ref{cor:moment-hierarchy}) but not
for $k>j$ (Theorem~\ref{thm:insuff} at order $j+1$).  The middle tier
corresponds to linear Bayes / prevision updating
(cf.\ Goldstein~(\citeyear{goldstein1975})).

\begin{remark}[Conformal row]\label{rem:conformal-row}
Conformal prediction achieves marginal coverage
$\Pbb(Y_{n+1}\in\hat{C}_n)\ge 1-\alpha$ under exchangeability without
specifying any moments.  The resulting prediction sets are finite-sample
valid and distribution-free, but set-valued; marginal coverage does not
imply correct specification of predictive probabilities.  This is a
qualitatively different guarantee from predictive completeness and the two
are not directly comparable.
\end{remark}

For a predictive characterization of exchangeable laws via mixing measures,
see Fortini and Petrone~(\citeyear{fortini2017predictive}).

\subsection{Closure theorem}

\begin{definition}[Predictive completeness]\label{def:pred-complete}
A martingale posterior is \emph{predictively complete} if it determines
$\Pbb(X_{n+1:n+k}=x_{1:k}\mid\mathcal{F}_n)$ for all $k\ge 1$ and all
cylinder events.
\end{definition}

\begin{theorem}[Closure]\label{thm:closure}
Let $(\theta_n)$ be a $[0,1]$-valued martingale satisfying
Assumption~\ref{ass:bounded}, with terminal value $\theta_\infty$.  The
following are equivalent:
\begin{enumerate}[label=(\roman*),nosep]
\item The martingale posterior determines
$\Ebb[(1-\theta_\infty)^k\mid\mathcal{F}_n]$ for all $k\ge 1$.
\item The conditional law of $\theta_\infty$ given $\mathcal{F}_n$ is
uniquely determined (a.s.).
\item The conditional moment sequence
$(\Ebb[\theta_\infty^j\mid\mathcal{F}_n])_{j\ge 1}$ is uniquely
determined (a.s.).
\end{enumerate}
\end{theorem}

\begin{proof}
(ii)$\Rightarrow$(iii): If the conditional law
$\mathcal{L}(\theta_\infty\mid\mathcal{F}_n)$ is uniquely determined,
then every measurable function of $\theta_\infty$ has a unique
conditional expectation.  In particular,
$\Ebb[\theta_\infty^j\mid\mathcal{F}_n]=\int\theta^j\dx\mathcal{L}(\theta_\infty\mid\mathcal{F}_n)$
is uniquely determined for each $j\ge 1$.

(iii)$\Rightarrow$(i): By the moment expansion
(Corollary~\ref{cor:moment-hierarchy}),
$\Ebb[(1-\theta_\infty)^k\mid\mathcal{F}_n]
=\sum_{j=0}^k\binom{k}{j}(-1)^j\Ebb[\theta_\infty^j\mid\mathcal{F}_n]$,
which is a polynomial in the conditional moments of orders
$1,\ldots,k$.  If these moments are uniquely determined, so is every
$k$-step run probability.

(i)$\Rightarrow$(ii): Suppose $\Ebb[(1-\theta_\infty)^k\mid\mathcal{F}_n]$
is uniquely determined for all $k\ge 1$.  By M\"obius inversion
(Proposition~\ref{prop:mobius}), the conditional moment sequence
$(\Ebb[\theta_\infty^j\mid\mathcal{F}_n])_{j\ge 1}$ is also uniquely
determined.  Since $\theta_\infty\in[0,1]$, the Hausdorff moment theorem
(Theorem~\ref{thm:hausdorff}) guarantees that a probability measure on
$[0,1]$ is uniquely determined by its moment sequence.  Therefore the
conditional law $\mathcal{L}(\theta_\infty\mid\mathcal{F}_n)$ is
uniquely determined a.s.

The equivalence (i)$\Leftrightarrow$(iii) also follows directly from
the invertibility of the M\"obius matrix
(Proposition~\ref{prop:mobius}), without invoking the Hausdorff theorem.
\end{proof}

A martingale posterior is predictively complete if and only if it
uniquely determines the conditional law of $\theta_\infty$ given
$\mathcal{F}_n$.  Condition~\eqref{eq:mp-condition} is necessary but not
sufficient: it fixes only $\Ebb[\theta_\infty\mid\mathcal{F}_n]=\theta_n$
and leaves the conditional law underdetermined.  On $[0,1]$, specifying
the law is equivalent to specifying all moments (Hausdorff determinacy;
Theorem~\ref{thm:hausdorff}), which is equivalent to specifying all
$k$-step predictives.

\section{Optimal stopping and dynamic amplification}\label{sec:stopping}

The moment hierarchy has direct consequences for sequential decision
problems.  Optimal stopping is the prototypical such problem: a
clinician deciding when to stop a trial, a firm choosing when to
exercise an option, or an analyst choosing when to ``peek'' at
accumulating data must each evaluate multi-step continuation values.
These values are polynomial functionals of the posterior and therefore
depend on moments beyond the mean.  The distortion introduced by
ignoring higher moments is not merely theoretical: it shifts the
optimal stopping boundary itself.

\subsection{Setup}\label{sec:stopping-setup}

Fix $K\ge 2$.  At time $n$, define the value of the multi-step
prediction problem under the full posterior law:
\begin{equation}\label{eq:value-full}
  V_n \;:=\; \sup_{2\le\tau\le K}\;\Ebb\!\left[(1-\theta_\infty)^\tau
  \mid\mathcal{F}_n\right].
\end{equation}
The constraint $\tau\ge 2$ restricts attention to genuinely multi-step
objectives; the case $\tau=1$ reduces to first-moment coherence already
determined by the martingale condition.  Under the mean-only plug-in,
replace $\theta_\infty$ by its conditional mean
$m_n=\Ebb[\theta_\infty\mid\mathcal{F}_n]$:
\begin{equation}\label{eq:value-plugin}
  \widetilde{V}_n \;:=\; \sup_{2\le\tau\le K}\;(1-m_n)^\tau.
\end{equation}

\subsection{Strict value distortion}\label{sec:stopping-prop}

\begin{proposition}[Exact value gap]\label{prop:value-gap}
Let $\sigma_n^2:=\Var(\theta_\infty\mid\mathcal{F}_n)$.  If\/
$\sigma_n^2>0$, then $V_n-\widetilde{V}_n=\sigma_n^2$.
\end{proposition}

\begin{proof}
Since $x\mapsto(1-x)^\tau$ is decreasing on $[0,1]$ for each fixed
$\tau\ge 1$, both suprema are attained at the lower boundary $\tau=2$.
By the second-moment identity
$\Ebb[(1-\theta_\infty)^2\mid\mathcal{F}_n]=(1-m_n)^2+\sigma_n^2$
(equation~\eqref{eq:k2-identity}), we have
$V_n=(1-m_n)^2+\sigma_n^2$ and $\widetilde{V}_n=(1-m_n)^2$.
\end{proof}

\begin{corollary}[Stopping boundary distortion]\label{cor:stopping}
For $r\in(0,1)$, define the effective prediction horizon under the full
posterior as $\tau^*_n(r):=\max\{2\le\tau\le K:
\Ebb[(1-\theta_\infty)^\tau\mid\mathcal{F}_n]\ge r\}$, and the
plug-in horizon as $\widetilde\tau^*_n(r):=\max\{2\le\tau\le K:
(1-m_n)^\tau\ge r\}$.  There exist states $\mathcal{F}_n$ and
thresholds $r$ for which $\tau^*_n(r)>\widetilde\tau^*_n(r)$.
\end{corollary}

\begin{proof}
By Proposition~\ref{prop:value-gap}, for each $\tau_0\in\{2,\ldots,K-1\}$,
$\Ebb[(1-\theta_\infty)^{\tau_0}\mid\mathcal{F}_n]>(1-m_n)^{\tau_0}$
when $\sigma_n^2>0$.  Choose $r$ strictly between these two values;
then $\tau_0$ lies in the Bayes feasible set but not the plug-in feasible set.
\end{proof}

The distortion $V_n-\widetilde{V}_n=\sigma_n^2$ vanishes at rate $O(1/n)$
as the posterior concentrates (Proposition~\ref{prop:bvm-rate}), so
predictive completeness is necessary for optimal sequential decisions.

\section{Discussion}\label{sec:discussion}

For one-step prediction ($k=1$), first-moment coherence suffices:
$\Pbb(X_{n+1}=1\mid\mathcal{F}_n)=\theta_n$.
For $k\ge 2$, the martingale condition alone does not uniquely identify
$k$-step predictives (Theorem~\ref{thm:insuff}), and the plug-in
$(1-\theta_n)^k$ is inadmissible (Corollary~\ref{cor:inadmissible}).
Predictive completeness requires the full conditional law of
$\theta_\infty\mid\mathcal{F}_n$ (Theorem~\ref{thm:closure}).

Many practical martingale-posterior constructions do in fact specify a full
conditional law for $\theta_\infty\mid\mathcal{F}_n$, in which case the
closure condition of Theorem~\ref{thm:closure} is automatically satisfied.

\begin{example}[Diagnostic application]\label{ex:diagnostic}
Consider a martingale posterior defined by the update
$\theta_{n+1}=\theta_n+\gamma_n(X_{n+1}-\theta_n)$ for a learning rate
sequence $\gamma_n\to 0$.  This satisfies~\eqref{eq:mp-condition} by
construction.  If $\gamma_n=1/(n+c)$ for some $c>0$, then
$\theta_n=(S_n+c\theta_0)/(n+c)$, which is the posterior mean under a
$\BetaD(c\theta_0,c(1-\theta_0))$ prior.  The conditional law of
$\theta_\infty$ is then $\BetaD(c\theta_0+S_n,\,c(1-\theta_0)+n-S_n)$:
fully specified, and all predictives are determined.

For other choices of $\gamma_n$ (such as $\gamma_n=n^{-\alpha}$ with
$\alpha\ne 1$), the limiting conditional law is typically not a
Beta distribution and may not be analytically available.  The
martingale condition still holds, but without the conditional law, the
practitioner cannot compute $\Ebb[(1-\theta_\infty)^k\mid\mathcal{F}_n]$
for $k\ge 2$.  In such cases, bootstrap or resampling approximations to
the conditional law are needed to recover multi-step predictive
coherence.
\end{example}

In classical Bayesian inference, the likelihood
$\theta^{S_n}(1-\theta)^{n-S_n}$ updates $\Pi$ to
$\Pi(\cdot\mid\mathcal{F}_n)$, determining all moments.
Goldstein~(\citeyear{goldstein1975uniqueness}) made this mechanism
explicit: when the posterior mean is linear in the data, the likelihood
moments $(m_i)$ bootstrap from the first prior moment to all higher ones
via a recursive identity.  That recursion is the structural ingredient
the martingale condition omits.  Equivalently, the likelihood/prior
specification supplies the typical-set geometry (KL/entropy rates) for
learning along a single path~\citep{polson2025prediction}.
Absent this specification (whether via likelihood or alternative
construction), multi-step predictives are not uniquely determined.

The moment hierarchy rests on the classical theory of moment problems
(Hausdorff~(\citeyear{hausdorff1921summationsmethoden});
Feller~(\citeyear{feller1971introduction})).  Hausdorff
determinacy (that a measure on $[0,1]$ is uniquely determined
by its moment sequence) is the key structural tool, used both to
establish injectivity of the $k$-step run probabilities
(Theorem~\ref{thm:injective}) and to prove the closure theorem
(Theorem~\ref{thm:closure}).  Compact support is essential: the moment
problem is generically indeterminate for unbounded models
(Remark~\ref{rem:real-line}), where analogous results remain open.

Fong, Holmes, and Walker~(\citeyear{fong2023martingale}) introduced the
martingale posterior as a general-purpose alternative to
likelihood-based updating.  Their framework specifies a predictive
resampling scheme that, in implementations, often amounts to
specifying a full conditional law for $\theta_\infty\mid\mathcal{F}_n$.
The insufficiency results of Section~\ref{sec:insuff} apply to the
abstract martingale condition, not to specific implementations that go
beyond it.  The closure theorem (Theorem~\ref{thm:closure}) and the
optimal stopping distortion (Section~\ref{sec:stopping}) together
characterize the full consequences of first-moment versus full-law
specification.

Goldstein and Wooff~(\citeyear{goldstein1998wooff}) provide a linear-Bayes
treatment of adjusting exchangeable beliefs; their framework deliberately
stops short of full distributional specification.
Theorem~\ref{thm:closure} identifies the predictive consequences of
that choice.  The related contributions of Polson and
Zantedeschi~(\citeyear{polson2025prediction}) and Datta and
Polson~(\citeyear{datta2025polynomial}) to finite-moment prediction and
empirical Bayes are placed in context in the Introduction.
Bayesian predictive inference via c.i.d.\ sequences, without a fully
specified prior, is developed in Berti, Dreassi, Pratelli and
Rigo~(\citeyear{berti2021class}); Fortini and
Petrone~(\citeyear{fortini2017predictive}) give a predictive
characterization of exchangeable laws in a related setting.
Leisen, Pratelli and Rigo~(\citeyear{leisen2025weak}) characterize
c.i.d.\ sequences via the martingale property of $(\alpha_n(f))$
for bounded Borel~$f$.  The martingale posterior of Fong, Holmes and
Walker~(\citeyear{fong2023martingale}) constrains only
$\alpha_n(\mathrm{id})=\theta_n$; our results quantify the
predictive consequences of that gap.
In a complementary direction, Battiston and
Cappello~(\citeyear{battiston2025beyond}) introduce \emph{almost
conditional identically distributed} (a.c.i.d.) sequences, in which
predictive distributions are measure-valued almost supermartingales
parametrized by a sequence of deviation parameters; where the present
results identify what the strict martingale condition underdetermines,
their framework broadens the class of admissible predictive schemes
beyond the c.i.d.\ setting.

Several questions remain open.  First, whether finitely many moment
constraints beyond the first yield ``approximate'' coherence in a
precise quantitative sense: the second-order bound
(Proposition~\ref{prop:quantitative}) controls $k$-step predictives up
to third-order remainders, but minimax rates over posteriors with given
$(\mu_1,\mu_2)$ are unknown.  Second, for exponential families with
non-compact parameter space (Gaussian mean, Poisson rate), the moment
problem is generically indeterminate (Remark~\ref{rem:real-line}), and
what replaces Hausdorff determinacy is open.  Third, whether
approximate methods (variational, ABC) can be disciplined by moment
constraints to achieve practical completeness.

Martingale coherence~\eqref{eq:mp-condition} constrains only the first
conditional moment.  Predictive completeness requires the full
conditional law, equivalently all moments on $[0,1]$
(Theorem~\ref{thm:closure}).  The Bayes posterior supplies this; the
martingale condition alone does not.

\subsection*{Acknowledgements}

We thank Luca Pratelli and Pietro Rigo for comments that helped
reshape several sections of this paper.

\begin{appendix}

\section{Hausdorff moment theorem}\label{app:hausdorff}

\begin{theorem}[Hausdorff]\label{thm:hausdorff}
A sequence $(\mu_k)_{k\ge 0}$ with $\mu_0=1$ is the moment sequence of a
probability measure on $[0,1]$ if and only if it is completely monotone:
$(-1)^m\Delta^m\mu_k\ge 0$ for all $k,m\ge 0$, where $\Delta$ is the
forward difference operator.  The moment sequence uniquely determines the
measure.
\end{theorem}

See Hausdorff~(\citeyear{hausdorff1921summationsmethoden}) and
Feller~(\citeyear{feller1971introduction}), Chapter~VII.

\section{General functional non-identification}\label{app:general-nonlinear}

\begin{theorem}[General nonlinearity]\label{thm:general-nonlinear}
Let $f:[0,1]\to\mathbb{R}$ be any non-affine continuous function.  For
any $m\in(0,1)$, there exist probability measures $\nu_1,\nu_2$ on
$[0,1]$ with $\int\theta\dx\nu_1=\int\theta\dx\nu_2=m$ but
$\int f\dx\nu_1\ne\int f\dx\nu_2$.
Therefore $\Ebb[f(\theta)\mid\mathcal{F}_n]$ is not determined by
$\Ebb[\theta\mid\mathcal{F}_n]$.
\end{theorem}

\begin{proof}
Since $f$ is non-affine and continuous, there exist $a<b$ in $[0,1]$
such that $f$ is not affine on $[a,b]$.  For $m\in(a,b)$, set
$\nu_1=\delta_m$ and $\nu_2=\lambda\delta_a+(1-\lambda)\delta_b$ with
$\lambda=(b-m)/(b-a)$.  Both have mean $m$.  By strict Jensen,
$\int f\dx\nu_2\ne f(m)=\int f\dx\nu_1$.  For $m\notin(a,b)$, a
three-point mixture with positive mass on $[a,b]$ achieves the same
conclusion by continuity and non-affinity of $f$.
\end{proof}

The case $f(\theta)=(1-\theta)^k$ recovers Lemma~\ref{lem:non-id}.
The case $f(\theta)=\ind[\theta\le t]$ shows that even the posterior CDF
at a single point is not determined by the mean; the case
$f(\theta)=-\theta\log\theta-(1-\theta)\log(1-\theta)$ (Bernoulli entropy,
strictly concave) gives the same conclusion for expected entropy.

\end{appendix}


\end{document}